\documentclass{amsart}
\usepackage[utf8]{inputenc}
\usepackage{setspace}
\usepackage[margin=1.25in]{geometry}
\usepackage{graphicx}
\graphicspath{ {./figures/} }
\usepackage{subcaption}
\usepackage{amsmath}
\usepackage{amssymb}
\usepackage{mathtools}
\usepackage{lineno}
\usepackage{amsfonts}
\usepackage{xfrac} 
\usepackage{graphicx} 
\usepackage{amsthm}
\usepackage[all]{xy}

\usepackage{float}
\usepackage{braket}

\usepackage{kbordermatrix}

\usepackage{multirow}
\usepackage{diagbox}
\usepackage{slashbox}

\usepackage[pagebackref=true]{hyperref}

\usepackage{tikz-cd}
\usetikzlibrary{cd}

\newtheorem{theorem}{Theorem}

\newtheorem{definition}[theorem]{Definition}
\newtheorem{lemma}[theorem]{Lemma}
\newtheorem{corollary}[theorem]{Corollary}
\newtheorem{remark}[theorem]{Remark}
\newtheorem{proposition}[theorem]{Proposition}

\newtheorem*{theorem*}{Theorem}

\title{Link Bundles of Compact Toric Varieties of Real Dimension 8}

\author{Shahryar Ghaed Sharaf}
\date{\today}


\begin{document}

	\begin{abstract}
The main goal of this work is to determine the Betti numbers of the links of isolated singularities in a compact toric variety of real dimension 8, using the CW-structure of the links. Additionally, we construct the intersection spaces associated with these links. Using the duality of the Betti numbers of intersection spaces, we conclude that, similar to the case of toric varieties of real dimension 6, the Betti numbers of the links contain only one non-combinatorial invariant parameter. In the final section, we extend our discussion to arbitrary compact toric varieties and their associated link bundles. We show that for any given link $\mathcal{L}$, there exists a fiber bundle $\pi: \mathcal{L} \to X$ with fiber $S^{1}$, where the base space $X$ is a compact toric variety. Furthermore, using the Chern–Spanier exact sequences for sphere bundles, we show that for the fiber bundle $\pi:\mathcal{L}\longrightarrow X$, where $\dim_{\mathbb{R}}(X)=6$, the non-combinatorial invariant parameters appearing in the Betti numbers of $\mathcal{L}$ and $X$ are equal. In addition, we provide an algebraic description of the non-combinatorial invariant parameter of $X$ in terms of the cohomological Euler class of the fiber bundle.
    \end{abstract}

	\maketitle

    \tableofcontents

	\section{Introduction}
 A toric variety is a complex algebraic variety containing an algebraic torus as an open dense subset, such that the action of the torus on itself extends to the whole variety. From the point of view of algebraic topology, compact toric varieties are pseudomanifolds with even dimensional strata. Let $Y$ be an 8-dimensional compact toric variety. Consider the stratification on $Y$ induced by the preimages of the faces of the underlying polytope under a moment map:
 \begin{align*}
 Y=Y_{8} \supset Y_{6} \supset Y_{4} \supset Y_{2} \supset Y_{0}.	
 \end{align*}
 We follow the construction of links introduced in \cite{banagl2024link}. The link of $Y_6 - Y_{4}$ is homeomorphic to an $S^1$. The study of the links of $Y_4- Y_{2}$ and $Y_2 - Y_{0}$ goes along the same line as the study of the links of the middle and bottom strata of a 6-dimensional compact toric variety. Hence, the links of the connected components of $Y_{4} - Y_{2}$ are rational homology 3-spheres, and the links of the connected components of $Y_{2} - Y_{0}$ satisfy Poincaré duality rationally. Our main object of study in this work is the link of an isolated singularity $y \in Y_{0}$, which we denote by $\mathcal{X}$. If a point $y \in Y_{0}$ in the bottom stratum of $Y$ is regular, its link is referred to as the link of the isolated point $y$. The topological space $\mathcal{X}$ is a 7-dimensional  pseudomanifold with possibly 1-dimensional rational singularities. The pseudomanifold $\mathcal{X}$ has only even co-dimensional strata:
 \begin{align*}
 	\mathcal{X}= \mathcal{X}_{7} \supset \mathcal{X}_{5} \supset \mathcal{X}_{3} \supset \mathcal{X}_{1},
 \end{align*}
 where the links of the strata $\mathcal{X}_{5} -\mathcal{X}_{3}$ and $\mathcal{X}_{3} - \mathcal{X}_{1}$ are $S^{1}$ and rational homology 3-spheres, respectively.
 
In the first part of this work, we endow the link $\mathcal{X}$ with a CW structure. We modify the method introduced in \cite{banagl2024link}, such that the resulting CW-structure on $\mathcal{X}$ is relatively simple, which allows it to be used to compute the Betti numbers. At the end of Section 3, we arrive at Proposition \ref{Hom1Ver}, an intermediate result.
 In Section 4, we generalize the theory of intersection spaces for non-isolated singularities introduced by Banagl in \cite{banagl2010intersection} to pseudomanifolds with two rational strata. In other words, we use a concept of $\mathbb{Q}$-homology stratified pseudomanifolds to form the intersection space based on non-isolated singularity techniques for two strata. As mentioned above, the links of the middle stratum of $\mathcal{X}$ are rational homology 3-spheres and therefore do not disturb Poincaré duality rationally. Our approach adapts Rourke and Sanderson's concept of homology stratification \cite{rourkesanderson}, extending its properties to our framework. Section 4 serves as an intermediate tool to simplify the results obtained in Section 3.
 
  There are several cohomology theories available that restore duality. The intersection cohomology $IH^{\ast}$ of Goresky and MacPherson \cite{goresky1980intersection}, the $L^2$-cohomology of Cheeger \cite{cheeger1980hodge}, \cite{cheeger1979spectral},
 \cite{cheeger1983spectral} for appropriately conical metrics on the regular part, and the homotopy-theoretic methods of intersection spaces $IX$ of Banagl \cite{banagl2010intersection} yielding a cohomology theory $HI^{\ast}(X) := H^{\ast}(IX)$. We apply the theory of intersection spaces to the link $\mathcal{X}$ to establish the duality of Betti numbers and show that the parameter $b_{1}$ defined in Proposition \ref{Hom1Ver} equals zero. Therefore, we reformulate Proposition \ref{Hom1Ver} and represent two main theorems of this work.
 \begin{theorem*}
 	Let $\mathcal{X}$ be the link of an isolated singularity in an 8-dimensional toric variety. Let $(M , \partial M)$ be the $\mathbb{Q}$-manifold obtained by cutting out all 1-dimensional singularities of $\mathcal{X}$ and $i: \partial M \xhookrightarrow[]{ \;\;\;\; } M  $ be the inclusion map. Then, the Betti numbers of $\mathcal{X}$ are
 	\begin{align*}
 	b^{\mathcal{X}}_{7}=1, \;	b^{\mathcal{X}}_{6}=0, \;	&b^{\mathcal{X}}_{5}=f_{1}-4,\;	b^{\mathcal{X}}_{4}=3f_{1}-f_{2}-6, \;
 	b^{\mathcal{X}}_{3}=3f_{1}-f_{2}-6-b_{2}, \\	&b^{\mathcal{X}}_{2}=f_{1}-4-b_{2}, \;	b^{\mathcal{X}}_{1}=0, \; b^{\mathcal{X}}_{0}=1,
 \end{align*}
 	where 	$b_{2} = \operatorname{rk}(\operatorname{ker}(H_{3}(\partial M) \xrightarrow{ \; \;  i_{\ast}  \; \; } H_{3} (M)  ))$, and $f_{1}$ and $f_{2}$ are the numbers of the 2-dimensional and 1-dimensional faces of the underlying polytope of $\mathcal{X}$, respectively. 
 \end{theorem*}
 \begin{theorem*}
 	Let $\mathcal{X}$ be the link of an isolated singularity in an 8-dimensional toric variety. Let $I^{\bar{n}}\mathcal{X}$ be the generalized intersection space associated with $\mathcal{X}$ with respect to middle perversity $\bar{n}$, as defined in Definition \ref{IXgen}. Then, the Betti numbers of $I^{\bar{n}}\mathcal{X}$ are given by
 \begin{align*}
 	b_{7}^{I^{\bar{n}}\mathcal{X}}&=0,\; 
 	b_{6}^{I^{\bar{n}}\mathcal{X}}=m-1, \;
 	b_{5}^{I^{\bar{n}}\mathcal{X}}=f_{2}-4-b_{2}, \;
 	b_{4}^{I^{\bar{n}}\mathcal{X}}=3f_{1}-f_{2}-6-b_{2}, \\
 	b_{3}^{I^{\bar{n}}\mathcal{X}}&=3f_{1}-f_{2}-6-b_{2},  \;
 	b_{2}^{I^{\bar{n}}\mathcal{X}}=f_{2}-4-b_{2}, \;
 	b_{1}^{I^{\bar{n}}\mathcal{X}}=m-1, \;
 	b_{0}^{I^{\bar{n}}\mathcal{X}}=1, 
 \end{align*}
 where $f_{1}$ and $f_{2}$ are the numbers of 2-dimensional and 1-dimensional faces of the underlying 3-dimensional polytope of $\mathcal{X}$, $\mathcal{P}$. The parameter $m$ denotes the number of rationally singular components of $\mathcal{X}_{1}$, and $b_{2}$ is defined in Equation (\ref{bp2}), or equivalently, in the above theorem.
 \end{theorem*}
In the final section, we turn our focus to the homotopy groups of links. We consider the link $\mathcal{L}$ of a point in an arbitrary compact toric variety. We show that there exists a fiber bundle $\pi : \mathcal{L} \longrightarrow X_{\Sigma}$ with fiber $S^{1}$, where $X_{\Sigma}$ is a compact toric variety associated with the complete fan $\Sigma$. We provide an explicit construction of the complete fan $\Sigma$. From this construction, it follows that the fan $\Sigma$ is not unique. We use the Chern–Spanier long exact sequence for homology groups of sphere bundles to relate the Betti numbers of $\mathcal{L}$ and $X_{\Sigma}$. It turns out that the parameter $b_{2}$ defined in the above theorem coincides with the non-combinatorial invariant parameter $b$ in the Betti numbers of $X_{\Sigma}$ defined by McConnell in \cite{mcconnell1989rational} or in \cite{banagl2024link}. The Chern–Spanier long exact sequence for homology groups also provides an algebraic description of the parameter $b$ using the cohomological Euler class of the fiber bundle.

\section{Preliminaries}\label{Preliminaries}
 In this work, we mainly use definitions and constructions from \cite{banagl2024link}. We will also use standard terminology for fans, cones, and their faces from \cite{cox2024toric}. However, to maintain completeness and integrity, we briefly provide some essential definitions and results without proof.
\begin{remark}
	 Throughout this paper, we work with rational coefficients unless otherwise specified. We denote links and underlying polytopes by calligraphic characters.
\end{remark}
    

\begin{definition}
	Let $\Sigma$ be a complete fan in $\mathbb{R}^n$. 
	Reversing the inclusions and the dimensional grading in the face lattice of $\Sigma$, we obtain 
	the face lattice of an \textit{abstract} polyhedron which can be realized in 
	$\mathbb{R}^{n}$ as a regular (polyhedral) cell complex. We denote this abstract 
	polyhedron by $\mathcal{P}(\Sigma)$ and its geometrical realization by 
	$\vert \mathcal{P}(\Sigma) \vert$. The abstract polyhedron $\mathcal{P}(\Sigma)$ 
	is called \textbf{ the dual polyhedron} or \textbf{the dual polytope}. If the complete fan 
	$\Sigma$ is understood we will simply write $\mathcal{P}$. By abuse of notation 
	we will frequently write $\mathcal{P}$ for $\vert \mathcal{P} \vert$. The dimension of $\mathcal{
		P}(\Sigma)$ is $n$, since the top dimensional cell of $\mathcal{
		P}(\Sigma)$ is associated to the cone $\{0\}$ and it is $n$-dimensional. 
\end{definition}

\begin{remark}\label{Bijection}
	The faces in $\mathcal{P}(\Sigma)$ are in one-to-one correspondence 
	(in complementary dimension) with the cones in $\Sigma$. This defines a 
	bijection $\delta: \mathcal{P} \longrightarrow \Sigma$ and we denote the dual 
	cone to $\tau \in \mathcal{P}$ by $\delta(\tau)\in \Sigma$. In particular $\Sigma$ can be reconstructed from $\mathcal{
		P}$. From that perspective we shall also write $\Sigma= \Sigma_{\mathcal{P}}$. 
\end{remark}

\begin{definition} 
	A CW complex is called \textbf{regular} if its characteristic maps can 
	be chosen to be embeddings.
\end{definition}

\begin{definition}{} 
	The 0-dimensional faces of $\mathcal{P}$ are called \textbf{vertices}. The 1-dimensional faces are called \textbf{edges} and \textbf{facets} are faces with co-dimension 1. The set
	$\mathcal{P}^{i} = \{ F: \: F \: \textit{is a face of} \: \mathcal{P} , \: \dim(F) \leq i \}$ is the i-skeleton of $\mathcal{P}$.
\end{definition}

\noindent \textbf{Regular CW structure of the dual polytope $\mathcal{P}$}. \\
Note that $\mathcal{P}$ is homeomorphic to an $n$-disc $\mathcal{D}^{n}$, 
where $\dim(\mathcal{P})=n$. The structure of the underlying fan induces a regular CW structure on $\mathcal{P}$ as follows: \\
The vertex $\{ 0 \} \in \Sigma$ is dual to the interior of $\mathcal{P}$, represented by an $n$-dimensional cell in the CW structure. Each $k$-dimensional cone in $\Sigma$ is dual to an $(n-k)$-dimensional face of $\mathcal{P}$, represented by an $(n-k)$-dimensional cell in the CW structure. Note also that the boundary maps are induced from the inclusion of faces in $\Sigma$ (or dually in $\mathcal{P}$). This gives us a CW complex on $\mathcal{P} \cong \mathcal{D}^{n}$, which we will use later.
Although we do not require the exact form of the boundary operators of $\mathcal{P}$ in our preceding discussion, we briefly describe the attaching maps. 
As usual, we start with a discrete set, $X^{0}$, whose points represent the vertices of $\mathcal{P}$. We glue each 1-dimensional cell to its topological boundary in $\mathcal{P}$, which consists of its neighboring 0-dimensional cells.
Inductively, we attach each $k$-cell, representing a $k$-dimensional face of $\mathcal{P}$, to its lower-dimensional neighboring faces. Note that $\vert \mathcal{P} \vert - \operatorname{int}(\mathcal{P}) \cong S^{n-1}$. In the last step we attach $\operatorname{int}(\mathcal{P})$, represented by the $n$-cell, to its topological boundary $\vert \mathcal{P} \vert - \operatorname{int}(\mathcal{P})$. Note that due to the fact that each $k$-dimensional face of $\mathcal{P}$ is homeomorphic to $\mathcal{D}^{k}$, each characteristic map can be chosen to be an embedding. Thus, the above CW structure is a regular CW structure.\\
We proceed with the definition of topological pseudomanifolds. Next, we briefly explain the collapsing of parallel subtori within a torus. Finally, we outline the construction of the links of an arbitrary compact toric variety.

\begin{definition}\label{pseudomanifolds}
	We define a \textbf{topologically stratified space} inductively on dimension. A 0-dimensional topologically stratified space $X$ is a countable set with the discrete topology. For $m > 0$ an \textbf{$m$-dimensional topologically stratified space} is a para-compact Hausdorff topological space $X$ equipped with a filtration
	\begin{align*}
		X=X_{m} \supseteq X_{m-1} \supseteq \dots \supseteq X_{1} \supseteq X_{0} \supseteq X_{-1}= \emptyset
	\end{align*}
	by closed subsets $X_{j}$ such that if $x \in X_{j}-X_{j-1}$ there exists a neighborhood $\mathcal{N}_{x}$ of $x$ in $X$, a compact $(m-j-1)$-dimensional topologically stratified space $\mathcal{L}$ with filtration
	\begin{align*}
		\mathcal{L}=\mathcal{L}_{m-j-1} \supseteq \dots \supseteq \mathcal{L}_{1} \supseteq  \mathcal{L}_{0} \supseteq \mathcal{L}_{-1} = \emptyset,
	\end{align*}  
	and a homeomorphism $\phi : \mathcal{N}_{x} \longrightarrow \mathbb{R}^{j} \times \mathcal{C}(\mathcal{L}),$
	where $\mathcal{C}(\mathcal{L})$ is the open cone on $\mathcal{L}$, such that $\phi$ takes $\mathcal{N}_{x} \cap X_{j+i+1}$ homeomorphically onto
	\begin{align*}
		\mathbb{R}^{j} \times \mathcal{C}(\mathcal{L}_{i}) \subseteq \mathbb{R}^{j} \times \mathcal{C}(\mathcal{L})
	\end{align*}
	for $m-j-1 \geq i \geq 0$, and $\phi$ takes $\mathcal{N}_{x} \cap X_{j}$ homeomorphically onto
	\begin{align*}
		\mathbb{R}^{j} \times \{ \text{vertex of }\; \mathcal{C}(\mathcal{L}) \}.
	\end{align*}
\end{definition}
    Note that the homeomorphism type of the stratified space $\mathcal{L}$ is not uniquely determined by the above definition.
\begin{definition}
	In Definition \ref{pseudomanifolds}, any stratified space $\mathcal{L}$ that satisfies the required properties is referred to as a \textbf{link} of the stratum at $x$.   
\end{definition}

\begin{remark}\label{linkandstratum}
	It follows that $X_{j}-X_{j-1}$ is a $j-$dimensional topological manifold. (The empty set is a manifold of any dimension.) We call the connected components of these manifolds the \textbf{strata} of $X$.
\end{remark}

\begin{definition}\label{PSMFD}
	An \textbf{$m$-dimensional topological pseudomanifold} is a para-compact Hausdorff topological space $X$ which possesses a topological stratification such that $X_{m-1}=X_{m-2}$
	and $X-X_{m-1}$ is dense in X.
\end{definition}

\begin{definition}\label{DefSing}
	We call a stratum (homologically) \textbf{singular} if none of its links is a homology sphere. A stratum is (homologically) \textbf{rationally singular} if none of its links is a rational homology sphere.
\end{definition}
Let $T^{n}=$\scalebox{1.3}{ $\sfrac{\mathbb{R}^{n}}{\mathbb{Z}^{n}}$}$ 
\cong \overbrace{ S^{1} \times \dots \times S^{1}}^{n}$ 
be an $n$-torus. The projection map $\pi: \mathbb{R}^{n} \longrightarrow  
\text{\scalebox{1.3}{ $\sfrac{\mathbb{R}^{n}}{\mathbb{Z}^{n}} $}}$ maps a 
\textbf{rational} $k$-dimensional linear subspace $\mathbf{V} \subset \mathbb{R}^{n}$ 
to a compact subtorus $\pi(\mathbf{V}) \subset T^{n}$. The \textbf{rationality} here 
means that $\mathbf{V}$ has a basis in $\mathbb{Z}^{n}$. Now each affine $k$-plane 
parallel to $\mathbf{V}$  in $\mathbb{R}^{n}$ determines a subtorus ``parallel'' to 
$\pi(\mathbf{V})$. Collapsing each of these parallel subtori to a point will 
give us a compact subspace $\text{\scalebox{1.3}{$\sfrac{T^{n}}{\pi(\mathbf{V})}$} } \subset T^{n}$, which is obviously homeomorphic to $T^{n-k}$.\\
Let $X$ be a toric variety over a polytope $\mathcal{
	P}$. Consider the natural projection map $X \xrightarrow{p }\mathcal{P}$. The preimage $X_{2i} = p^{-1}(\mathcal{P}^{i})$ is a $2i$-dimensional topological space. The filtration
\begin{align}\label{strtTV}
	X_{\mathcal{P}}=X_{2m} \supset X_{2(m-1)} \supset \dots \supset X_{2} \supset X_{0},
\end{align}
is a stratification of the toric variety $X_{\mathcal{P}}$.

\noindent\textbf{Construction of links}\label{Triv} \phantom{.} \\ Let $\tau$ be a face of $\mathcal{P}$. Let $\mathcal{M}_{\tau}$ be an abstract polytope, geometrically realized as a subspace of $\mathcal{P}$, and defined as follows: \\
Let $\mathcal{S}_{\tau}=\big\{ \sigma \vert \sigma \in \mathcal{P}, \; \sigma \cap \tau \neq \emptyset \; \text{and} \; \dim(\sigma) > \dim(\tau) \big\}$ be the set of all higher dimensional neighboring faces of $\tau$ in $\mathcal{P}$. \\
We shall define an abstract polytope $\mathcal{M}_{\tau}$ with the following properties: 
\begin{enumerate}
	\item $\forall \sigma \in \mathcal{S}_{\tau}: \exists! \; \gamma_{\sigma} \in \mathcal{M}_{\tau} \; \text{such that} \; \operatorname{int}(\gamma_{\sigma}) \cap \operatorname{int}(\sigma) \neq \emptyset \; \text{with} \; \dim(\gamma_{\sigma})= \dim(\sigma)-(1+\dim(\tau))$ and $\operatorname{int}(\gamma_{\sigma}) \cap \operatorname{int}(\sigma^{\prime})= \emptyset$ if $\sigma^{\prime} \in \mathcal{P}$ and $\sigma^{\prime} \neq \sigma$.
	\item If $\omega, \sigma \in \mathcal{S}_{\tau}$ with $\omega \prec \sigma$, then we demand that $\gamma_{\omega} \prec \gamma_{\sigma}$ and $\gamma_{\sigma} \cap \omega = \gamma_{\omega}$.
\end{enumerate}
Note that $\operatorname{int}(\sigma) \cong \operatorname{int}(\mathcal{D}^{\dim(\sigma)})$. So, the first requirement can be satisfied by embedding $\operatorname{int}(\mathcal{D}^{\dim(\gamma_{\sigma})})$ in $\operatorname{int}(\mathcal{D}^{\dim(\sigma)})$. Yet the second condition describes how the boundary of $\mathcal{D}^{\dim(\gamma_{\sigma})}$ meets the boundary of $\mathcal{D}^{\dim(\sigma)}$. 
We refer to $\mathcal{M}_{\tau}$ as a \textbf{base} of the link of $x \in p^{-1}(\operatorname{int}(\tau)) \subset X_{\mathcal{P}}$. \\
\begin{proposition}\label{Existence2}
	Let $\Sigma$ be a complete fan in $\mathbb{R}^n$ and $\mathcal{
		P}$ its dual polytope. Given any face $\tau \in \mathcal{P}$ there exists an abstract polytope $\mathcal{M}_{\tau}$ satisfying the above conditions $(1)$ and $(2)$ and having a geometrical realization $\vert \mathcal{M}_{\tau} \vert$ as a subpolyhedron of a geometrical realization of $\mathcal{P}$. 
\end{proposition}
\begin{proof}
	The proof can be found in \cite[Proposition 14]{banagl2024link}
\end{proof}
Now, let $\vert \mathcal{M}_{\tau} \vert \subset \vert \mathcal{P} \vert$ be a geometrical realization of $\mathcal{M}_{\tau}$ and $n=\dim(\vert \mathcal{P} \vert)$. Recall the bijection $\delta: \mathcal{P} \longrightarrow \Sigma_{\mathcal{P}}$ that we introduced earlier. For each $\gamma \in \mathcal{M}_{\tau}$ choose $\sigma_{\gamma} \in \mathcal{S}_{\tau}$ such that $\gamma \cap \sigma_{\gamma} \neq \emptyset$. Note $\tau \prec \sigma_{\gamma}$ and thus $\delta (\sigma_{\gamma}) \prec \delta(\tau)$ in $\Sigma_{\mathcal{P}}$. This implies that $\pi(\delta (\sigma_{\gamma})) \subset \pi(\delta(\tau))$ in $T^{n}$. We construct $\mathcal{L}_{\tau}$, the link of a point $x \in p^{-1}(\operatorname{int}(\tau))$, by means of the following map: \\
\begin{align}
	\mathcal{M}_{\tau} \times T^{n} &\longrightarrow \mathcal{L}_{\tau} \nonumber \\
	\{y\} \times T^{n}  &\longmapsto \{y\} \times \scalebox{1.3}{ $\sfrac{T^{n}}{\Big(\pi (\delta(\sigma_{\gamma}))\times \big( \sfrac{T^{n}}{\pi (\delta(\tau))} \big) \Big)}$},
\end{align}\label{Linkmap}
where $y \in \operatorname{int}(\gamma)$.
\begin{remark}\label{StratProof}
	There is a natural projection $\mathcal{L}_{\tau} \xrightarrow{\:\:\: p_{\mathcal{L}} \:\:\:} \mathcal{M}_{\tau}$. Similarly, we define $\big( \mathcal{L}_{\tau} \big)_{2i+1}=p_{\mathcal{L}}^{-1}(\mathcal{M}^{i}_{\tau})$.
\end{remark}
\begin{proposition}\label{trivlink}
	Let $\Sigma$ be a complete fan and $\mathcal{P}$ the associated dual polytope. Then $X_{\mathcal{P}}$, the associated toric varieties with $\mathcal{P}$, is a topological pseudomanifold with a trivial link bundle. 	
\end{proposition}
\begin{proof}
	The proof can be found in \cite[Proposition 19]{banagl2024link}
\end{proof}
\begin{remark}
	Given $\tau , \eta \in \mathcal{P}$ with $\dim(\tau)=\dim(\eta)$, 
	$\mathcal{L}_{\tau}$ and $\mathcal{L}_{\eta}$ are not necessarily even homotopy equivalent.
\end{remark}

\section{CW structure of links }

We employ our previous methods in \cite{banagl2024link} to compute the homology groups of the link of an isolated singularity, constructed as above, in an 8-dimensional toric variety. The underlying polytope of the link $\mathcal{X}$ is a 3-dimensional polytope, $\mathcal{M}$. We endow $\mathcal{M}$ with the regular CW structure induced by face relations. Let $p_{\mathcal{X}}: \mathcal{X} \longrightarrow \mathcal{M}$ be the natural projection. Observe that Diagram
\begin{align}
	\xymatrix{
		\mathcal{M} \times T^4 \ar[r] \ar[rd] & \mathcal{X} \ar[d]^{p_{\mathcal{X}}} \\ & \mathcal{M}
	}
\end{align}
commutes, where the map $\mathcal{M}\times T^4 \longrightarrow \mathcal{M}$ is the projection onto the first component, $\mathcal{M}$. Note that the decomposition $\mathcal{X}_{2i+1}= p^{-1}_{\mathcal{X}}(\mathcal{M}^{i})$ constitutes a stratification of $\mathcal{X}$. Let $T^2$ be the torus associated with a 1-dimensional face of $\mathcal{M}$, denoted by $\tau$. Since the polytope $\mathcal{M}$ is homeomorphic to a disc $D^3$, the face $\tau$ has two neighboring 2-dimensional faces, $\sigma_{1}$ and $\sigma_{2}$, with the tori $T^{3}_{1}$ and $T^{3}_{2}$ attached, respectively. Hence, we have the following commutative diagram:
\begin{align}\label{DIAT}
	\xymatrix{
		T^{4} \ar[d] \ar[r] & T^{3}_{1} \ar[d] \\
		T^{3}_{2} \ar[r] & T^{2}.
	} 
\end{align}
Each map in the diagram above is a collapsing map, where we extract the collapsing data from the 4-dimensional cone dual to the isolated point. We aim to endow $T^3_1$, $T^3_2$, and $T^2$ with CW structures such that all maps in the above diagram are cellular. Note that replacing the collapsing maps in $\mathcal{X}$ with homeomorphic cellular maps—while ensuring the above diagrams remain commutative for each 1-dimensional face of $\mathcal{M}$—does not alter the homeomorphism type of the resulting space. To be more precise, we aim to replace the collapsing maps in the link up to homeomorphism, so that the resulting space is homeomorphic to the link. Consider the quotient map $\pi: \mathbb{R}^{4} \longrightarrow \sfrac{\mathbb{R}^{4}}{\mathbb{Z}^{4}} \cong T^{4}$. We consider the minimal CW structure on the 4-dimensional torus:
\begin{align*}
	T^{4} \cong e^{4} \cup \big( \bigcup_{i=1}^{4} e^{3}_{z_{i}} \big) \cup \big( \bigcup_{i=1}^{6} e^{2}_{y_{i}} \big) \cup \big( \bigcup_{i=1}^{4} e^{1}_{x_{i}} \big) \cup e^{0}.
\end{align*}

 Let $T^{3}_{i} \cong \sfrac{T^{4}}{S^1_{i}}$, where $S^1_{i} \cong \pi (\begin{pmatrix}
 	n_{i} &
 	m_{i} &
 	l_{i} &
 	k_{i}
 \end{pmatrix}^{T}) $, $n_{i}$, $m_{i}$, $l_{i}$, and $k_{i}$ are all non-zero, and $\operatorname{gcd}(n_{i},m_{i},l_{i},k_{i})=1$ for $i=1,2$. We employ the machinery introduced in \cite[Section 4.2]{banagl2024link} to endow $T^3_1$ and $T^3_2$ with CW structures. The resulting CW structure on $T^3_{i}$ is
\begin{align*}
	T^{3}_{i} \cong \big( \bigcup_{j=1}^{3} e^{3}_{j} \big) \cup \big( \bigcup_{j=1}^{9} e^{2}_{j} \big) \cup \big( \bigcup_{j=1}^{7} e^{1}_{j} \big) \cup e^{0}
\end{align*}
 for $i=1,2$.
We illustrate the above CW structures in Figure \ref{T3CWFig}.
\begin{figure}[h!]
	\centering
	\includegraphics[width=.35\linewidth]{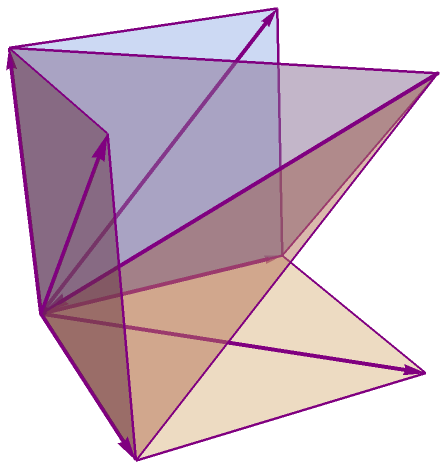}
	\includegraphics[width=.35\linewidth]{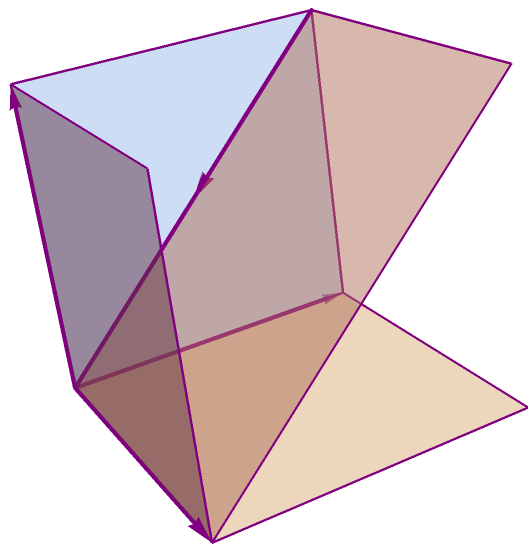}
	\caption{CW structures on $T^{3}$}
	\label{T3CWFig}
\end{figure}
One could also use the previous method to get a CW structure on $T^2$. The resulting CW structure on $T^2$ is rather complicated.

The main idea of this section is to define a homeomorphism $T^{2} \longrightarrow T^{2}$ such that the resulting CW structure on the target torus is simpler compared to the CW structure on the domain $T^{2}$, which is obtained using the method introduced in \cite[Section 4.2]{banagl2024link}. The construction proceeds as follows. Consider the projection map $ \operatorname{proj}_{i}: \mathbb{R}^{4} \longrightarrow \mathbb{R}^{3}_{i} $, where we project onto the hyperplane orthogonal to $ (n_{i}, m_{i}, l_{i}, k_{i})^{T} \in \mathbb{R}^{4} $. Let $ \overline{e^{1}_{T^{3}_{1_{\xi}}}} \cong S^{1}_{1_{\xi}} $ and $ \overline{e^{1}_{T^{3}_{2_{\xi}}}} \cong S^{1}_{2_{\xi}} $ be the closures of 1-cells of $ T^3_{1} $ and $ T^3_{2} $, respectively. Let \( \pi_{3_{i}}: \mathbb{R}^{3}_{i} \longrightarrow \sfrac{\mathbb{R}^{3}_{i}}{\mathbb{Z}^{3}} \cong T^{3}_{i} \). Consider the 1-dimensional subspaces \( \pi^{-1}_{3_i}(S^{1}_{i_{\alpha}}) \cong \mathbb{R}^{1}_{3i_{\alpha}} \subset \mathbb{R}^{3}_{i} \). Note that we have the projection maps $ \operatorname{proj}_{ij}: \mathbb{R}^{3}_{i} \longrightarrow \mathbb{R}^{2} $ for $ i=1,2 $, where we project onto the 2-dimensional subspace \( \mathbb{R}^{2} \subset \mathbb{R}^{3}_{i} \) orthogonal to \( v_{j} = \operatorname{proj}_{i}((n_{j}, m_{j}, l_{j}, k_{j})) \in \mathbb{R}^{3}_{i} \), such that \( i \neq j \). Consider the 1-dimensional subspaces \( \mathbb{R}_{i\alpha} = \operatorname{proj}_{ij}(\mathbb{R}^{1}_{3i_{\alpha}}) \subset \mathbb{R}^{2} \) for \( i=1,2 \). We choose \( \mathbb{R}_{i \alpha} \) and \( \mathbb{R}_{j \beta} \) such that the angle between the two subspaces is less than \( \pi \), with all other \( \mathbb{R}_{k \gamma} \) lying between them. Let \( R \) be the set of all such 1-dimensional subspaces. Note that under the map \( \pi_{2}: \mathbb{R}^{2} \longrightarrow \sfrac{\mathbb{R}^{2}}{\mathbb{Z}^{2}} \cong T^{2} \), each \( \mathbb{R}_{l \theta} \in R \) corresponds to an \( S^{1}_{l \theta} \). We equip \( T^{2} \) with a CW structure where each 1-cell lies on \( \pi_{2}(\mathbb{R}_{l \theta}) \) for \( \mathbb{R}_{l \theta} \in R \), and the intersection points \( \pi_{2}(\mathbb{R}_{l \theta}) \cap \pi_{2}(\mathbb{R}_{l^{\prime} \theta^{\prime}}) \) for \( l \neq l^{\prime} \) and \( \theta \neq \theta^{\prime} \) represent the 0-cells. This CW structure on $T^{2}$ makes the maps \( T^{4} \longrightarrow T^{2} \), \( T^{3}_{1} \longrightarrow T^{2} \), and \( T^{3}_{2} \longrightarrow T^{2} \) cellular. However, the resulting CW structure on \( T^{2} \) remains quite complex for our intended purposes. In the upcoming discussion, we aim to simplify the previous CW structure. Let \( \pi_{2}(\mathbb{R}_{i \alpha}) \cong S^{1}_{h} \) and \( \pi_{2}(\mathbb{R}_{j \beta}) \cong S^{1}_{v} \). The images of the remaining 1-dimensional subspaces \( \mathbb{R}_{k \gamma} \) are \( S^{1} \) that warp around the torus \( T^{2} \). Let \( e^{0}_{0} \) be the image of the origin under the map \( \pi_{2} \). In other words, let \( e^{0}_{0} \) be the 0-cell that has a non-empty intersection with all \( \pi_{2}(\mathbb{R}_{l \theta}) \). Let \( L_{h} \subset S^{1}_{h} \) be the line that connects \( e^{0}_{0} \) with one of the two next 0-cells that lie on \( S^{1}_{h} \). From the construction, one can deduce that the following procedure does not depend on the choice of the 0-cells in this step. We have \( \overline{S^{1}_{h} \setminus L_{h}} \cong I_{h} = [0,1] \). Consider the cylinder \( S^{1}_{v} \times \overline{S^{1}_{h} \setminus L_{h}} \subset T^{2} \). We define the map \( p_{h}: T^{2} \longrightarrow \sfrac{T^{2}}{\sim} \) such that \( S^{1}_{v} \times \overline{S^{1}_{h} \setminus L_{h}} \cong S^{1}_{v} \times I_{h} \sim S^{1}_{h} \times \{0\} \) and the identity elsewhere. In other words, we project the cylinder \( S^{1}_{v} \times \overline{S^{1}_{h} \setminus L_{h}} \) to one of the components of its boundary, \( S^{1}_{v} \times \{0\} \). Consequently, we arrive at \( \sfrac{T^{2}}{\sim} \cong (T^{\prime})^{2} \), a torus. Let \( S^{1}_{h^{\prime}} = p_{h}(S^{1}_{h}) \). Choose a 0-cell next to \( e^{0}_{0} \) that lies on \( S^{1}_{v} \subset (T^{\prime})^{2} \). Let \( L_{v} \subset S^{1}_{v} \) be the line that connects the previous 0-cells. Consider the cylinder \( S^{1}_{h^{\prime}} \times \overline{S^{1}_{v} \setminus L_{v}} \cong S^{1}_{h^{\prime}} \times I_{v} = S^{1}_{h^{\prime}} \times [0,1] \). We define the projection map \( p_{v}: (T^{\prime})^{2} \longrightarrow \sfrac{(T^{\prime})^{2}}{\sim} \) such that \( S^{1}_{h^{\prime}} \times \overline{S^{1}_{v} \setminus L_{v}} \cong S^{1}_{h^{\prime}} \times I_{v} \sim S^{1}_{h^{\prime}} \times \{0\} \) and the identity elsewhere. It follows that \( \sfrac{(T^{\prime})^{2}}{\sim} \cong (T^{\prime \prime})^{2} \), a torus. Let \( S^{1}_{v^{\prime}} = p_{v}(S^{1}_{v}) \). Choose the lines \( \widetilde{L}_{v} \subset S^{1}_{v^{\prime}} \) and \( \widetilde{L}_{h} \subset S^{1}_{h^{\prime}} \) as follows. The lines \( \widetilde{L}_{v} \subset S^{1}_{v^{\prime}} \) and \( \widetilde{L}_{h} \subset S^{1}_{h^{\prime}} \) connect \( e^{0}_{0} \) with the next 0-cell on \( S^{1}_{v^{\prime}} \) and \( S^{1}_{h^{\prime}} \), respectively. We define \( p_{L}:(T^{\prime \prime})^{2} \longrightarrow \sfrac{(T^{\prime \prime})^{2}}{\sim} \) such that \( \overline{S^{1}_{v^{\prime}} \setminus \widetilde{L}_{v}} \sim e^{0}_{0} \), \( \overline{S^{1}_{h^{\prime}} \setminus \widetilde{L}_{h}} \sim e^{0}_{0} \), and the identity elsewhere. From the construction, it follows that \( \sfrac{(T^{\prime \prime})^{2}}{\sim} \cong T^{2} \), a torus. The resulting CW structure on \( T^{2} \) has \( \# R \) 1-cells, \( (\# R-1) \) 2-cells, and a 0-cell. It follows that the maps \( T^{4} \longrightarrow T^{3}_{i} \longrightarrow T^{2} \xrightarrow{p_{L} \circ p_{v} \circ p_{h}} T^{2} \) are cellular for \( i=1,2 \). Thus, we have
\begin{align}\label{T2CW}
	T^{2} \cong \big( \bigcup_{i=1}^{\# R-1} e^{2}_{i} \big) \cup \big( \bigcup_{i=1}^{\# R} e^{1}_{i} \big) \cup e^{0}.
\end{align}
We illustrate the above procedures in Figure \ref{T2Modif}.

 \begin{remark}
 	
The above procedure is one of the cornerstones of this work. Therefore, in this remark, our goal is to provide a schematic elaboration for each step of the construction mentioned above. Consider Figure \ref{T2Modif}. 
\begin{figure}[H]
	\centering
	\includegraphics[width=.26\linewidth]{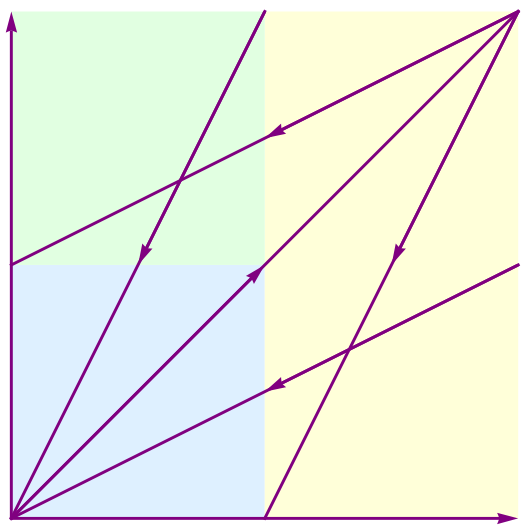}
	\includegraphics[width=.14\linewidth]{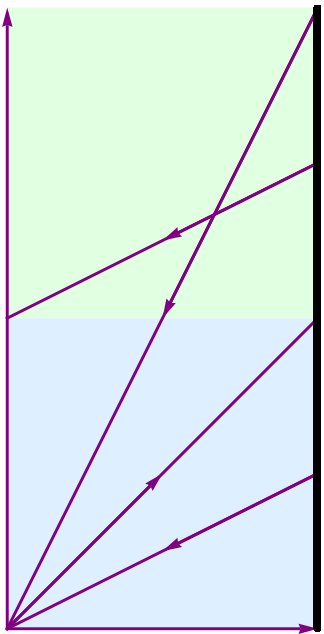}
	\includegraphics[width=.20\linewidth]{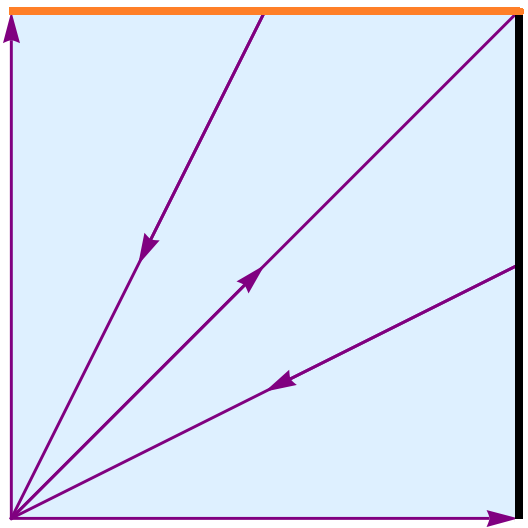}
	\includegraphics[width=.20\linewidth]{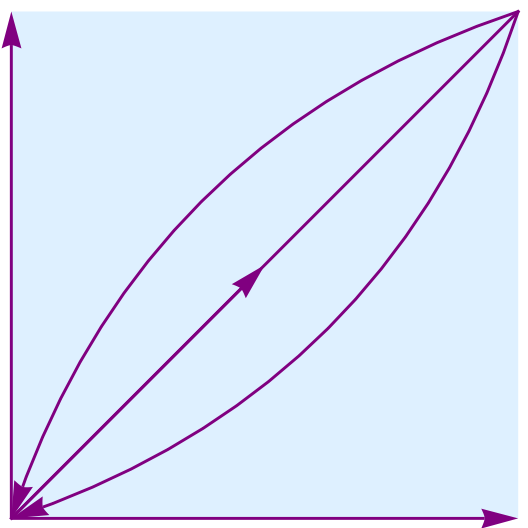}
	\caption{Schematic description of the construction of the map $T^2 \longrightarrow T^2$.}
	\label{T2Modif}
\end{figure}
The left illustration represents a CW structure on $T^2$, where we identify the opposite sides of the large square that contains the green square, the blue square, and the yellow rectangle. Each line represents an $S^1 \in T^2$. The yellow rectangle represents the cylinder $S^1_{v} \times \overline{S^{1}_{h} \setminus L_{h}}$. Under the projection $p_{h}$, the cylinder $S^1_{v} \times \overline{S^{1}_{h} \setminus L_{h}}$ maps to the circle denoted by the black line in the second illustration from the left. The green square in the same illustration represents the cylinder $S^{1}_{h^{\prime}} \times \overline{S^{1}_{v} \setminus L_{v}}$, which maps to the circle denoted by the orange line in the third illustration. In the last step, we map the upper half of the black line, $\widetilde{L}_{v}$, and the right half of the orange line, $\widetilde{L}_{h}$, to the top corner of the blue square, $e^{0}_{0}$.
\end{remark}

However, we should study the other possible CW structures on $T^3$. Let $T^3 \cong \sfrac{T^{4}}{S^{1}}$, where $S^{1}= \pi(v)$ for $v \in \mathbb{R}^{4}$ such that $v$ has one or two zero entries and the non-zero entries are pairwise relatively prime. The induced CW structure on $T^3$ is the following
\begin{align*}
	T^{3} \cong \big( \bigcup_{j=1}^{2} e^{3}_{j} \big) \cup \big( \bigcup_{j=1}^{5} e^{2}_{j} \big) \cup \big( \bigcup_{j=1}^{4} e^{1}_{j} \big) \cup e^{0}
\end{align*}
Schematically, we illustrate the previous CW structure in Figure \ref{T3CWFig}.
\begin{remark}\label{Alpha}
Compute the number of 1-cells, $\alpha_{i}$, needed for each $T^2_{i}$ to make the maps in Diagram \ref{DIAT} cellular. Let $\alpha=\operatorname{max} \{ \alpha_{i} \}$. We endow each $T^2_{i}$ with $\alpha$ 1-cells, $(\alpha - 1)$ 2-cells, and a 0-cell. Hence, we can make all the collapsing maps cellular. This consideration is made merely for the sake of simplicity in the following computation. We can also equip each $T^2_{i}$ with the minimal number of 1-cells required, namely $\alpha_{i}$. In this case, the discussion follows the same line as well. In Theorem \ref{MainTheo2}, we provide an algebraic description of the parameter $b_{2}$, which is of interest. Hence, the above consideration does not affect the main result of this work; it is merely an intermediate tool for the proof of the main theorems.
\end{remark}
 Let $\alpha$ be as described in the above remark. The above consideration encompasses all the required CW structures on $T^3$ and $T^2$ for our calculations. By simplifying the CW structure on each $T^2$ through the above procedure, the most general cellular chain groups for the link $\mathcal{X}$ is given by 

	\begin{align*}
	\mathcal{C}_{7} (\mathcal{X}) &= \mathbb{Q} \braket{e^{4}_{T^{4}} \times e^{3}_{\mathcal{M} } } \\
	\mathcal{C}_{6} (\mathcal{X}) &= \bigoplus_{i=1}^{4} \mathbb{Q} \braket{e^{3}_{T^{4}_{i} } \times e^{3}_{\mathcal{M} } } \\
	\mathcal{C}_{5} (\mathcal{X}) &= \bigoplus_{i=1}^{6} \mathbb{Q} \braket{e^{2}_{T^{4}_{i} } \times e^{3}_{\mathcal{M} } } \bigoplus_{j=1}^{\omega } \bigoplus_{i=1}^{3} \mathbb{Q} \braket{ e^{3}_{T^{3}_{\omega_{i}} } \times e^{2}_{\mathcal{M}_{j}}} \bigoplus_{j=1}^{\gamma } \bigoplus_{i=1}^{2} \mathbb{Q} \braket{ e^{3}_{T^{3}_{\gamma_{i}} } \times e^{2}_{\mathcal{M}_{j}} } \bigoplus_{j=1}^{\eta } \mathbb{Q} \braket{ e^{3}_{T^{3}_{\eta_{i}} } \times e^{2}_{\mathcal{M}_{j}}}  \\
	\mathcal{C}_{4} (\mathcal{X}) &= \bigoplus_{i=1}^{4} \mathbb{Q} \braket{e^{1}_{T^{4} } \times e^{3}_{\mathcal{M} } } \bigoplus_{j=1}^{ \omega} \bigoplus_{i=1}^{9}  \mathbb{Q} \braket{ e^{2}_{T^{3}_{\omega_{i} } } \times e^{2}_{\mathcal{M}_{j} } }  \bigoplus_{j=1}^{ \gamma} \bigoplus_{i=1}^{5} \mathbb{Q} \braket{ e^{2}_{T^{3}_{\gamma_{i} } } \times e^{2}_{\mathcal{M}_{j} } }  \bigoplus_{j=1}^{ \eta} \bigoplus_{i=1}^{3} \mathbb{Q} \braket{ e^{2}_{T^{3}_{\eta_{i} } } \times e^{2}_{\mathcal{M}_{j}} } \\
	\mathcal{C}_{3} (\mathcal{X}) &=   \mathbb{Q} \braket{ e^{0}_{T^{4} } \times e^{3}_{\mathcal{M} }  }  \bigoplus_{j=1}^{\omega} \bigoplus_{i=1}^{7}   \mathbb{Q} \braket{ e^{1}_{ T^{3}_{\omega_{i}  } } \times e^{2}_{\mathcal{M}_{j} }  } \bigoplus_{j=1}^{\gamma} \bigoplus_{i=1}^{4}  \mathbb{Q} \braket{ e^{1}_{ T^{3}_{\gamma_{i}  } } \times e^{2}_{\mathcal{M}_{j} }   } \bigoplus_{j=1}^{\eta} \bigoplus_{i=1}^{3}  \mathbb{Q} \braket{e^{1}_{ T^{3}_{\eta_{i}  } } \times e^{2}_{\mathcal{M}_{j} } } \\  &\phantom{=}  \bigoplus_{j=1}^{f_{2}} \bigoplus_{i=1}^{\alpha-1} \mathbb{Q} \braket{ e^{2}_{T^{2}_{i  }  } \times e^{1}_{\mathcal{M}_{j} }  } \\ 
	\mathcal{C}_{2} (\mathcal{X}) &= \bigoplus_{i=1}^{f_{1}} \mathbb{Q} \braket{e^{0}_{T^{3}} \times e^{2}_{\mathcal{M}_{i} }    } \bigoplus_{j=1}^{f_{2}} \bigoplus_{i=1}^{\alpha} \mathbb{Q} \braket{ e^{1}_{T^{2}_{ i  }  } \times e^{1}_{\mathcal{M}_{j}  } }  \\
	\mathcal{C}_{1} (\mathcal{X}) &= \bigoplus_{i=1}^{f_{2}} \mathbb{Q} \braket{ e^{0}_{T^{2}} \times e^{1}_{\mathcal{M}_{i}  }   }  \bigoplus_{j=1}^{f_{2} -f_{1}+2} \mathbb{Q} \braket{ e^{1}_{S^{1}  }  \times e^{0}_{\mathcal{M}_{j}  }  } \\
	\mathcal{C}_{0} (\mathcal{X}) &= \bigoplus_{i=1}^{f_{2}-f_{1}+2} \mathbb{Q} \braket{ e^{0}_{\mathcal{M}_{i}}  }, 
\end{align*}
where we labeled the tori $T^3$ with nine, five, and three 2-cells with $\omega$, $\gamma$, and $\eta$, respectively. Using the cellular chain groups above, we determine the boundary operators of the chain complex $\mathcal{C}_{\ast}(\mathcal{X})$ and subsequently compute the rational homology groups of the link $\mathcal{X}$. Computing $	\operatorname{rk} (H_{7}(\mathcal{X};\mathbb{Q})$ and $	\operatorname{rk} (H_{6}(\mathcal{X};\mathbb{Q})$ is straightforward, and we get
    \begin{align*}
	\operatorname{rk} (H_{7}(\mathcal{X};\mathbb{Q}) )&=1,\\
	\operatorname{rk}(H_{6}(\mathcal{X};\mathbb{Q}) )&=0.
	\end{align*}
   Before continuing our discussion, we look at the chain complexes on $T^3_{\omega}$ and $T^3_{\gamma}$.  
     \begin{align*}
     	\mathbb{Z}^{3} \xrightarrow{ \partial_{3}^{T^{3}_{\omega}}  } \mathbb{Z}^{9}  \xrightarrow{\partial_{2}^{T^{3}_{\omega}}} \mathbb{Z}^{7} \xrightarrow{\partial_{1}^{T^{3}_{\omega}}} \mathbb{Z} \\
     	\mathbb{Z}^{2} \xrightarrow{ \partial_{3}^{T^{3}_{\gamma}}  } \mathbb{Z}^{5}  \xrightarrow{\partial_{2}^{T^{3}_{\gamma}}} \mathbb{Z}^{4} \xrightarrow{ \partial_{1}^{T^{3}_{\gamma}}} \mathbb{Z},
     \end{align*}
   where $\operatorname{rk(ker}\partial_{3}^{T^{3}_{\omega}})=1$, $\operatorname{rk(im}\partial_{3}^{T^{3}_{\omega}})=2$, $\operatorname{rk(ker}\partial_{2}^{T^{3}_{\omega}})=5$, $\operatorname{rk(im}\partial_{2}^{T^{3}_{\omega}})=4$ , $\operatorname{rk(ker}\partial_{3}^{T^{3}_{\gamma}})=7$, $\operatorname{rk(im}\partial_{3}^{T^{3}_{\gamma}})=0$, $\operatorname{rk(ker}\partial_{2}^{T^{3}_{\gamma}})=4$, and $\operatorname{rk(im}\partial_{2}^{T^{3}_{\gamma}})=1$. In the following discussion, we will encode the collapsing data in matrices $\mathsf{A}_{i}$, $\mathsf{B}_{i}$, $\mathsf{C}_{i}$, $\mathsf{D}_{i}$, and $\mathsf{E}_{i}$. Since we are working with rational coefficients, we do not need the explicit form of these matrices. Hence, we arrive at the following boundary operators.
    \begin{tiny}
	\begin{align*}
		\partial_{5}=\kbordermatrix{  & e^{2}_{T^{4}_{1} \times e^{3}_{ \mathcal{M} } } &  e^{2}_{T^{4}_{2} \times e^{3}_{ \mathcal{M} } }  &  e^{2}_{T^{4}_{3} \times e^{3}_{ \mathcal{M} } }  &  e^{2}_{T^{4}_{4} \times e^{3}_{ \mathcal{M} } }  &  e^{2}_{T^{4}_{5} \times e^{3}_{ \mathcal{M} } }  &  e^{2}_{T^{4}_{6} \times e^{3}_{ \mathcal{M} } }  & e^{3}_{T^{3}_{ \omega_{j}}} \times e^{2}_{ \mathcal{M}_{\omega} } & e^{3}_{T^{3}_{ \gamma_{j}}} \times e^{2}_{ \mathcal{M}_{\gamma} } & e^{3}_{T^{3}_{ \eta_{j}}} \times e^{2}_{ \mathcal{M}_{\eta} } \\
	e^{1}_{T^{4}_{1} } \times e^{3}_{\mathcal{M}} \phantom{\;\;  f_{1} \Big\{ }	& 0 & 0 & 0 & 0 & 0 & 0 &  \overbrace{0_{1 \times 3} }^{ \omega} & \overbrace{0_{1 \times 2}}^{ \gamma} & \overbrace{0}^{ \eta } \\ 
e^{1}_{T^{4}_{2} } \times e^{3}_{\mathcal{M}} \phantom{\;\;  f_{1} \Big\{ }	 	& 0 & 0  & 0 & 0 & 0 & 0 & 0_{1 \times 3} & 0_{1 \times 2} & 0 \\ 
   e^{1}_{T^{4}_{3} } \times e^{3}_{\mathcal{M}}  \phantom{\;\;  f_{1} \Big\{ }  	& 0 & 0 & 0 & 0 & 0 & 0 & 0_{1 \times 3} & 0_{1 \times 2} & 0 \\ 
    e^{1}_{T^{4}_{4} } \times e^{3}_{\mathcal{M}} \phantom{\;\;  f_{1} \Big\{ }   & 0 & 0 & 0 & 0 & 0 & 0 & 0_{1 \times 3} & 0_{1 \times 2} & 0 \\
     e^{2}_{T^{3}_{ \omega_{j}}} \times e^{2}_{ \mathcal{M}_{\omega} } \; \;  \omega \Big\{  & (\mathsf{A}_{1})_{9 \times 1} & (\mathsf{A}_{2})_{9 \times 1} & (\mathsf{A}_{3})_{9 \times 1} & (\mathsf{A}_{4})_{9 \times 1} & (\mathsf{A}_{5})_{9 \times 1} & (\mathsf{A}_{6})_{9 \times 1} & \partial^{T^{3}_{ \omega_{j} }}_{3} & 0_{9 \times 2} & 0_{9 \times 1} \\
       e^{2}_{T^{3}_{ \gamma_{j}}} \times e^{2}_{ \mathcal{M}_{\gamma} } \; \;  \gamma \Big\{ & (\mathsf{A}_{7})_{5 \times 1}  & (\mathsf{A}_{8})_{5 \times 1} & (\mathsf{A}_{9})_{5 \times 1} & (\mathsf{A}_{10})_{5 \times 1} & (\mathsf{A}_{11})_{5 \times 1} & (\mathsf{A}_{12})_{5 \times 1} & 0_{3 \times 5} & \partial^{ T^{3}_{ \gamma_{j} } }_{3} & 0_{5 \times 1} \\
       e^{2}_{T^{3}_{ \eta_{j}}} \times e^{2}_{ \mathcal{M}_{\eta} } \; \;  \eta \Big\{ & (\mathsf{A}_{13})_{3 \times 1}  & (\mathsf{A}_{14})_{3 \times 1} & (\mathsf{A}_{15})_{3 \times 1} & (\mathsf{A}_{16})_{3 \times 1} & (\mathsf{A}_{17})_{3 \times 1} & (\mathsf{A}_{18})_{3 \times 1} & 0_{3 \times 3} & 0_{3 \times 2} & 0_{3 \times 1}\\ }.
	\end{align*}	
\end{tiny}
      Considering the boundary operators of $T^{3}_{\omega}$ and $T^{3}_{\gamma}$, it follows that
      \begin{align*}
      	\operatorname{rk(ker}(\partial_{5})) = \gamma + \omega + \eta = f_{1}, \, \text{and} \, \operatorname{rk(im}(\partial_{5}))= 6 + \gamma + 2 \omega. 
      \end{align*}
  A straightforward computation shows that $\operatorname{rk(im}\partial_{6})=4$. Hence, we get
  \begin{align*}
  	\operatorname{rk}(H_{5}(\mathcal{X}; \mathbb{Q}))=f_{1}-4.
  \end{align*}
  Accordingly, we arrive at the following boundary operators:
	  \begin{align*}
	  \partial_{4} = \kbordermatrix{
		& e^{1}_{T^{4}_{1}} \times e^{3}_{\mathcal{M}}    & e^{1}_{T^{4}_{2}} \times e^{3}_{\mathcal{M}} & 	e^{1}_{T^{4}_{3}} \times e^{3}_{\mathcal{M}} & e^{1}_{T^{4}_{4}} \times e^{3}_{\mathcal{M}} &  e^{2}_{T^{3}_{ \omega_{j}}} \times e^{2}_{ \mathcal{M}_{\omega} }  &  e^{2}_{T^{3}_{ \gamma_{j}}} \times e^{2}_{ \mathcal{M}_{\gamma} } &  e^{2}_{T^{3}_{ \eta_{j}}} \times e^{2}_{ \mathcal{M}_{\eta} }  \\
		e^{0}_{T^{4}} \times e^{3}_{\mathcal{M}} \phantom{\;\;  f_{1} \Big\{ } & 0  & 0  & 0   & 0  & \overbrace{0_{1 \times 9 }  }^{ \omega}  & \overbrace{ 0_{1 \times 5 }  }^{   \gamma }  & \overbrace{ 0_{ 1 \times 3  }  }^{  \eta}    \\
		e^{1}_{ T^{3}_{ \omega_{j}   } } \times e^{2}_{\mathcal{M}_{ \omega  } } \;\;   \omega \Big\{   & (\mathsf{B}_{1})_{ 7 \times 1 }   &  (\mathsf{B}_{2})_{ 7 \times 1 }   & (\mathsf{B}_{3})_{ 7 \times 1 }   & (\mathsf{B}_{4})_{ 7 \times 1 }  & \partial_{2}^{T^{3}_{\omega_{j} } }  & 0_{4 \times 5 }  & 0_{4 \times 5 }    \\
		e^{1}_{ T^{3}_{ \gamma_{j}    } } \times e^{2}_{\mathcal{M}_{  \gamma  } } \;\;   \gamma \Big\{   & (\mathsf{B}_{5})_{ 4 \times 1 }   & (\mathsf{B}_{6})_{ 4 \times 1 }  & (\mathsf{B}_{7})_{ 4 \times 1 }   & (\mathsf{B}_{8})_{ 4 \times 1 }   & 0_{4 \times 6}   & \partial_{2}^{T^{3}_{\gamma_{j} } }   & 0_{4 \times 3}    \\
		e^{1}_{ T^{3}_{ \eta_{ j}    } } \times e^{2}_{\mathcal{M}_{\eta} } \;\;   \eta \Big\{ &  (\mathsf{B}_{9})_{ 3 \times 1 } & (\mathsf{B}_{10})_{ 3 \times 1 }  &  (\mathsf{B}_{11})_{ 3 \times 1 }  & (\mathsf{B}_{12})_{ 3 \times 1 }  &  0_{ 3 \times 6} & 0_{3 \times 5}  & 0_{3 \times 3 }    \\
		e^{2}_{T^{2}_{i } } \times e^{1}_{\mathcal{M}_{j} } \;\;  f_{2}  \Big\{ & 0_{(\alpha -1) \times 1}  & 0_{(\alpha -1) \times 1}  &  0_{(\alpha -1) \times 1}  & 0_{(\alpha -1) \times 1}  & (\mathsf{C}_{1})_{ (\alpha -1) \times 6 }  & (\mathsf{C}_{2})_{ (\alpha -1) \times 5 }  & (\mathsf{C}_{3})_{ (\alpha -1) \times 3 }     \\}
	\end{align*}
    and
	\begin{align*}
	\partial_{3} = \kbordermatrix{
		& e^{0}_{\mathcal{T }^{4}} \times e^{3}_{\mathcal{M}}  & e^{1}_{T^{3}_{1}} \times e^{2}_{\mathcal{M}_{i} }  & 	e^{1}_{T^{3}_{2}} \times e^{2}_{\mathcal{M}_{i} } & e^{1}_{T^{3}_{3}} \times e^{2}_{\mathcal{M}_{i} } & e^{2}_{T^{2}_{ i  }  } \times e^{1}_{\mathcal{M}_{ j }} &  \\
		e^{0}_{T^{3}} \times e^{2}_{\mathcal{M}_{k} } \;\; f_{1} \Big\{ &  \partial^{\mathcal{M}}_{3} & \overbrace{0_{1 \times 7}  }^{  \omega}  & \overbrace{0_{1 \times 4} }^{  \gamma}  &  \overbrace{0_{1 \times 3} }^{  \eta } &  \overbrace{0_{1 \times (\alpha-1)} }^{  f_{2} }  &  \\
		e^{1}_{T^{2}_{ i  } } \times e^{1}_{\mathcal{M}_{j} } \;\;  f_{2} \Big\{ & 0_{ \alpha \times 1 }  & (\mathsf{D}_{1})_{ \alpha \times 7 }  & (\mathsf{D}_{2})_{ \alpha \times 4 }  & (\mathsf{D}_{3})_{ \alpha \times 3 }  &  \partial^{ T^{2}_{j }}_{2}  &  \\}.
	\end{align*}
    Note that $\partial_{3} \circ \partial_{4}=0$. Then, for a given matrix $C_{i}$, where $i=1, 2 , 3$, the rows of the matrix $C_{i}$ are identical up to multiplication by -1. Using $\operatorname{rk(im}\partial_{2}^{T^{3}_{\omega}})=4$ and $\operatorname{rk(im}\partial_{2}^{T^{3}_{\gamma}})=1$, we arrive at $\operatorname{rk(im}(\partial_{4} ))= \gamma+ 4\omega + f_{2} + b_{1} + 4$. Consequently, we get 
    \begin{align}\label{bp1}
    \operatorname{rk(ker}(\partial_{4} ))= 5 \omega + 4 \gamma +3 \eta - (f_{2}+b_{1})=  5 \omega+4 \gamma +3 \eta -f_{2}-b_{1},
    \end{align}
    where $b_{1}$ determines the number of rows in $\partial_{4}$ which are linearly dependent. As a result, we arrive at
    \begin{align*}
    \operatorname{rk}(H_{4}(\mathcal{X}; \mathbb{Q}))=	3f_{1}-f_{2}-6-b_{1}.
    \end{align*}
    By considering the introduced CW complexes in Equality (\ref{T2CW}), we get $\operatorname{rk(ker(} \partial_{3}))= 4 \gamma + 7 \omega+ 3 \eta - ( f_{2}  + b^{ \prime})$ and $\operatorname{rk(im(} \partial_{3}))= \alpha f_{2} + b^{\prime}$, where $b^{\prime}$ determines the number of rows in $\partial_{3}$ which are linearly dependent. Finally, we determine the boundary operator $\partial_{2}$.
	\begin{align*}
	\partial_{2} = \kbordermatrix{ & e^{0}_{T^{3}} \times e^{2}_{\mathcal{M}_{l}}  & e^{1}_{T^{2}_{i } } \times e^{1}_{\mathcal{M}_{ j } }  \\
		e^{0}_{T^{2}} \times e^{1}_{\mathcal{M}_{j}} \;\; \phantom{-f_{1}+2} f_{2}  \Big\{ & \overbrace{\partial_{2}^{\mathcal{M}} }^{f_{1}}  & \overbrace{0_{1 \times \alpha} }^{ f_{2}}  \\
		e^{1}_{\mathcal{S}^{1} } \times e^{0}_{\mathcal{M}_{k} } \;\; f_{2}-f_{1}+2 \Big\{ & 0 & (\mathsf{E}_{1})_{1 \times \alpha} \\}
	\end{align*}
	It follows that $\operatorname{rk(ker(} \partial_{2}))=(\alpha f_{2})-(f_{2} - f_{1} +2)$.
    As a consequence, we arrive at 
    $\operatorname{rk}(H_{2}(\mathcal{X}; \mathbb{Q})) = (\alpha f_{2} ) - (f_{2} - f_{1} +2) - (\alpha f_{2} + b^{\prime} ) = f_{1}-f_{2}-2 -b^{\prime}.$
	From Poincar\'e duality, it follows that for a non-singular link $\operatorname{rk}(H_{2}(\mathcal{X};\mathbb{Q}))=\operatorname{rk}(H_{5}(\mathcal{X};\mathbb{Q}))=f_{1}-4$. Motivated by this observation, we set
	$
	\operatorname{rk}(H_{2}(\mathcal{X};\mathbb{Q}))=f_{1}-4-b_{2},$
	where $b_{2}=0$ for non-singular links. It yields
	\begin{align}\label{bp2}
	-b^{\prime}=f_{2}-2-b_{2}.
	\end{align}
    As a result, we arrive at $
    	\operatorname{rk}(H_{3}(\mathcal{X};\mathbb{Q}))=3 f_{1}-f_{2}-6-(b_{1}+b_{2}).$
At this point, we can present the main result of this section. 
\begin{proposition}\label{Hom1Ver}
	Let $\mathcal{X}$ be the link of an isolated point in a compact 8-dimensional toric variety. Then, we have
	\begin{align*}
	\operatorname{rk}(H_{7}(\mathcal{X};\mathbb{Q}) )&=1,\;
	\operatorname{rk}(H_{6}(\mathcal{X};\mathbb{Q}) )=0,\;
	\operatorname{rk}(H_{5}(\mathcal{X};\mathbb{Q}) )=f_{1}-4,\;
	\operatorname{rk(H_{4}(\mathcal{X};\mathbb{Q}) )}=3f_{1}-f_{2}-6-b_{1},\\
	\operatorname{rk}(H_{3}(\mathcal{X};\mathbb{Q}) )&=3f_{1}-f_{2}-6-(b_{1}+b_{2}),\;
	\operatorname{rk}(H_{2}(\mathcal{X};\mathbb{Q}))=f_{1}-4-b_{2},\;
	\operatorname{rk}(H_{1}(\mathcal{X};\mathbb{Q}) )=0,\\
	\operatorname{rk}(H_{0}(\mathcal{X};\mathbb{Q}) )&=1,
	\end{align*}
	where the parameters $b_{1}$ and $b_{2}$ are defined in Equations (\ref{bp1}) and (\ref{bp2}), respectively.
\end{proposition}

\begin{remark}\label{b1zero}
	Later on, by using the theory of intersection spaces, we will show $b_{1}=0$.
\end{remark}

\begin{corollary}
	The link of an isolated point in a compact 8-dimensional toric variety, $\mathcal{X}$, is a rational homology sphere if and only if $f_{1}=4$ and $f_{2}=6$. 
\end{corollary}\label{HomSphere}
\begin{proof}
Using the stratification introduced in \cite{banagl2024link}, we obtain the following stratification on \( \mathcal{X} \):
\begin{align*}
	\mathcal{X} = \mathcal{X}_{7} \supset \mathcal{X}_{3} \supset \mathcal{X}_{1},
\end{align*}
where 
\[
\mathcal{X}_{1} \cong \bigsqcup_{i=1}^{f_{2}-f_{1}+2} S^{1}.
\]
To show that the link bundle of $ \mathcal{X} $, considered as a pseudomanifold, is trivial, we follow a similar line of reasoning as in \cite{banagl2024link}. Accordingly, one can easily demonstrate that the link of a connected component of $ \mathcal{X}_{1} $ has the same CW structure as the link of an isolated singularity in a 6-dimensional toric variety. More precisely, this means that the underlying polytope of the link of a point $x \in S^{1} \subset \mathcal{X}_{1}$ is 2-dimensional. To its interior, we attach a $T^{3}$, while to each of its 1-dimensional and 0-dimensional faces, we attach a $T^{2}$ and an $S^{1}$, respectively. The collapsing data are derived from the cone corresponding to the isolated point in the real 8-dimensional toric variety for which $\mathcal{X}$ serves as the link of this point. This allows us to easily determine whether a connected component of $\mathcal{X}_{1}$ is, at least rationally, singular or regular. If $f_{1} = 4$ and $f_{2} = 6$, then the underlying polytope of the link is a tetrahedron. Hence, all points $x \in \mathcal{X}_{1}$ are rationally regular, and the link $\mathcal{X}$ satisfies Poincaré duality rationally. It follows that $b_{1} = 0$ and $b_{2} = 0$.
	
\end{proof}

\section{Intersection space of links}

In \cite{banagl2010intersection}, Banagl introduces the theory of intersection spaces for compact stratified pseudomanifolds with two strata. In \cite{banagl2024link}, we generalized the theory of intersection spaces to $\mathbb{Q}$-pseudomanifolds with isolated singularities. In this section, we aim to expand the theory of intersection spaces to $\mathbb{Q}$-pseudomanifolds with two strata. After some preparations, we show that the main ingredients of the proof of Theorem 2.47 in \cite[Section 2.9]{banagl2010intersection} remain intact in the case of $\mathbb{Q}$-pseudomanifolds with two strata. We then take the links of isolated singularities in an 8-dimensional toric variety as our objects of study and apply the generalized theory of intersection spaces for $\mathbb{Q}$-pseudomanifolds with two strata to them. As a result, we find that the parameter $b_{1}$ introduced in Equality (\ref{bp1}) equals zero, as mentioned in Remark \ref{b1zero}. We start with the definition of $\mathbb{Q}$-pseudomanifolds.

\begin{definition}[$\mathbb{Q}$-pseudomanifold]
	We call a topological space $X$ with filtration
	\begin{align}\label{Qstra}
		X \supset X_{i} \supset  X_{i-1} \supset \dots \supset X_{0} \supset X_{-1}= \emptyset
	\end{align}
	a $\mathbb{Q}$-\textbf{pseudomanifold} if 
	(\ref{Qstra}) can be augmented to a filtration of the form
	\begin{align*}
		X \supset X_{i+k}  \supset \dots \supset X_{i+1} \supset  X_{i} \supset \dots \supset X_{0} \supset X_{-1}= \emptyset,
	\end{align*} 
	such that $X$ is a pseudomanifold with respect to it and the 
	link of each connected component of $X_{i+(j+1)}-X_{i+j}$ for 
	$j=0, \dots, k-1$ in $X$ is a rational homology sphere. 
	We call the filtration (\ref{Qstra}) a 
	\textbf{$\mathbb{Q}$-homology stratification} 
	(or \textbf{$\mathbb{Q}$-stratification} for short) 
	of $X$.
\end{definition}

We need to show that the links $\mathcal{X}$ are $\mathbb{Q}$-pseudomanifolds with two strata.
\begin{lemma}
	Let $\mathcal{X}$ be the link of an isolated point in an 8-dimensional toric variety. Then, $\mathcal{X}$ is a $\mathbb{Q}$-pseudomanifold with the $\mathbb{Q}$-stratification
	\begin{align}
		\mathcal{X}= \mathcal{X}_{7} \supset \mathcal{X}_{7-6}
	\end{align}
\end{lemma}
\begin{proof}
	 As mentioned in the proof of Corollary \ref{HomSphere}, $\mathcal{X}$ is a pseudomanifold with the following stratification
	 \begin{align*}
	 	\mathcal{X}=\mathcal{X}_{7} \supset \mathcal{X}_{3} \supset \mathcal{X}_{1}.
	 \end{align*}
	 Following the introduced construction of the links in \cite[Section 3]{banagl2024link}, we can show that the link of the top stratum is a homological 3-sphere. Hence, the claim follows.
\end{proof}

\begin{definition}\label{Cout}
	We call the process of removing a disjoint union of
	cone-like neighborhoods of $\mathcal{X}_{1}$ \textbf{cutting out} the 1-dimensional stratum of $X$.
\end{definition}
Let $(M, \partial M)$ be the $\mathbb{Q}$-manifold obtained by cutting out 1-dimensional singularities of $\mathcal{X}$. As an immediate result of \cite[Theorem 47]{banagl2024link}, it follows that $(M, \partial M)$ satisfies Lefschetz duality rationally.

\begin{definition}\label{nseq}
	Let $n$ be a natural number. A CW complex $\mathcal{K}$ is called \textbf{rationally $n$-segmented} if it contains a sub-complex $\mathcal{K}_{<n} \subset \mathcal{K}$ such that
	$
	H_{r}(\mathcal{K}_{<n})=0 \;\; \text{for} \; r \geq n \;
	\text{and} \;
	i_{\ast}:H_{r}(\mathcal{K}_{<n}) \xrightarrow{\;\;\cong\;\;} H_{r}(\mathcal{K}) \; \text{for} \; r <n,
	$
	where $i$ is the inclusion of $\mathcal{K}_{<n}$ into $\mathcal{K}$.
\end{definition}
\begin{definition}(Goresky-MacPherson.)	
	A \textbf{perversity} 
	\begin{align*}
		\bar{p}: \mathbb{Z}_{ \geq 2} \longrightarrow \mathbb{Z}
	\end{align*}
	is a function such that
	$\bar{p}(2)=0 \; \text{and} \;
	\bar{p}(k+1)- \bar{p}(k) \in \{1,0\}.	$
	The \textbf{complementary perversity} $\bar{q}$ of $\bar{p}$ 
	is the one with $\bar{p}(k)+\bar{q}(k)=k-2$.
\end{definition}
Given any CW complex $\mathcal{K}$ and natural number $n$, there exists a \textbf{homology $n$-truncation} (Moore approximation) $f:\mathcal{K}_{<n} \to\mathcal{K}$, i.e., a continuous map from a CW complex $\mathcal{K}_{<n}$ to $\mathcal{K}$ such that $f_{\ast}:H_{\ast}(\mathcal{K}_{<n}; \mathbb{Z}) \to H_{\ast}(\mathcal{K}; \mathbb{Z})$ is an isomorphism for $\ast <n$ and $H_{\ast}(\mathcal{K}_{<n}; \mathbb{Z}) = 0$ for $\ast \geq n$, see \cite{banagl2010intersection} (for the simply connected case) and Wrazidlo \cite{wrazidlo2013induced} (in general). The map $f$ cannot, in general, be taken to be a sub-complex inclusion. We shall prove in Proposition \ref{LinkSeg} that the 5-dimensional links of the bottom stratum of $\mathcal{X}$ are rationally 3-segmented.

Let $X \supset X_{n-c}$ be an $n$-dimensional compact oriented topological pseudomanifold with two strata and a trivial link bundle. Similar to \cite[Section 6]{banagl2024link}, we can show that, for $x \in X_{n-c}$, although two different links need not be homeomorphic, the homology of links is well-defined.

Assume that $X$ has a trivial link bundle, that is, a neighborhood of a connected component of $X_{n-c} = \sqcup_{i} (X_{n-c})_{i}$ in $X$ looks like $(X_{n-c})_{i} \times \operatorname{cone}(\mathcal{L}_{i})$, where $\mathcal{L}_{i}$ is a $(c-1)$-dimensional closed manifold, a link of $(X_{n-c})_{i}$. For $x_{i} \in (X_{n-c})_{i}$, let $\mathcal{L}_{i}$ be an associated link. Assume that all links $\mathcal{L}_{i}$ can be equipped with CW structures such that they are rationally $k$-segmented, where $k=c-1-\bar{p}(c)$, for the perversity $\bar{p}$. Let $(M, \partial M)$ be the $\mathbb{Q}$-manifold with boundary obtained by cutting out all $(n-c)$-dimensional singularities of $X$. Let $(\mathcal{L}_{i})_{<k}$ be a sub-complex of $\mathcal{L}_{i}$ that truncates the homology. Then, we have
$
\partial M= \bigsqcup_{i} (X_{n-c})_{i} \times \mathcal{L}_{i}.
$
Let
$
X_{n-c} \times \mathcal{L}_{<k}= \bigsqcup_{i} (X_{n-c})_{i} \times (\mathcal{L}_{i})_{<k}
$
and define a homotopy class $g$ by the composition
$
g:X_{n-c} \times \mathcal{L}_{<k} \xhookrightarrow{\phantom{-}\phantom{-}} \partial M \xhookrightarrow{\phantom{-} \operatorname{incl.}  \phantom{-}} M$.
For the purposes of the present paper, we adopt the following definition:

\begin{definition}\label{IXgen}
	The perversity $\bar{p}$ rational intersection space $I^{\bar{p}}X$ of $X$ is defined to be
	\begin{align*}
		I^{\bar{p}}X=\operatorname{cone}(g)=M \bigcup_{g} \mathcal{C}((X_{n-c})_{i} \times \mathcal{L}_{<k}).
	\end{align*} 
\end{definition}
Due to the use of rational coefficients, $I^{\bar{p}}X$ as defined above need not be integrally homotopy equivalent to the construction of \cite{banagl2010intersection}, but the rational homology groups are isomorphic. 

\begin{theorem}\label{IX}
	Let $X \supset X_{n-c} = \sqcup_{i} (X_{n-c})_{i}$ be an $n$-dimensional, compact, oriented, stratified $\mathbb{Q}$-pseudomanifold with one singular stratum $X_{n-c}$ of dimension $(n-c)$ and trivial link bundle. The
	link $\mathcal{L}= \sqcup_{i} \mathcal{L}_{i}$ is assumed to be rationally $(c-1-\bar{p}(c))$ and $(c-1-\bar{q}(c))$-segmented and $X$, $X_{n-c}$ and $\mathcal{L}$ are oriented compatibly.
	Let $I^{\bar{p}}X$ and $I^{\bar{q}}X$ be $\bar{p}$- and $\bar{q}$-intersection spaces of $X$ with $\bar{p}$ and $\bar{q}$ complementary
perversities. Then there exists a generalized Poincaré duality isomorphism
		\begin{align*}
		D: \widetilde{H}_{r}(I^{\bar{p}}X)^{\ast} \xrightarrow{\phantom{-} \cong \phantom{-}} \widetilde{H}_{n-r}(I^{\bar{q}}X),
	\end{align*}
	where
	\begin{align*}
		\widetilde{H}_{r}(I^{\bar{p}}X)^{\ast} = \operatorname{Hom}(\widetilde{H}_{r}(I^{\bar{p}}X),\mathbb{Q}).
	\end{align*}

\end{theorem}
\begin{proof}
We mainly mimic the proof of \cite[Theorem 2.47]{banagl2010intersection}. Let $(M, \partial M)$ be the $\mathbb{Q}$-manifold obtained by cutting out the $n-c$-dimensional singularities of $X$. Note that \cite[Theorem 47]{banagl2024link} implies that $(M, \partial M)$ satisfies Lefschetz duality rationally, as mentioned. The topological space $\partial M$ is a $\mathbb{Q}$-manifold by \cite[Proposition 50]{banagl2024link}, and therefore, it rationally satisfies Poincar\'e duality using \cite[Corollary 49]{banagl2024link}. Hence, we can use \cite[Lemma 2.45]{banagl2010intersection} and \cite[Lemma 2.46]{banagl2010intersection} with rational coefficients. Consequently, we obtain a similar commutative diagram as in the proof of \cite[Theorem 2.47]{banagl2010intersection} with rational coefficients.

\end{proof}

\begin{proposition}\label{LinkSeg}
Let $\mathcal{X}$ be the 7-dimensional link of an isolated singularity in an 8-dimensional toric variety, with a $\mathbb{Q}$-stratification of the form $\mathcal{X}_{7} \supset \mathcal{X}_{1}$. For each $x_{i} \in \mathcal{X}_{1}$, let $\mathcal{L}_{i}$ denote the associated link. Then $\mathcal{L}_{i}$ is rationally 3-segmented.
\end{proposition}
\begin{proof}
The link of a point $x_{i} \in \mathcal{X}_{1}$, $\mathcal{L}_{i}$, has the same structure as the link of an isolated singularity in a 6-dimensional toric variety. More specifically, this means the following. The underlying polytope of $\mathcal{L}_{i}$, $\mathcal{M}_{i}$, is a 2-dimensional polytope. We attach a torus $T^3$ to the interior of the polytope $\mathcal{M}_{i}$, and each 1-dimensional face and vertex of $\mathcal{M}_{i}$ is associated with a torus $T^2$ and a circle $S^1$, respectively. Recall that we equipped each $T^{2}$ with the following CW structure:
\begin{align*}
	T^{2} \cong \big( \bigcup_{i=1}^{\alpha-1} e^{2}_{i} \big) \cup \big( \bigcup_{i=1}^{\alpha} e^{1}_{i} \big) \cup e^{0},
\end{align*}
where $\alpha$ is defined in Remark \ref{Alpha}. From each $T^{2}$, we omit one 2-cell. Despite omitting the 2-cells from the tori, the 2-skeleton of the link has not changed. The rest of the proof proceeds along the same lines as in \cite[Proposition 67]{banagl2024link}.

\end{proof}
Given that the link of $x \in \mathcal{X}_{1}$ is 3-segmented, we can apply Theorem \ref{IX} to the link $\mathcal{X}$.
\begin{lemma}
	Let $\mathcal{X}$ be the link of an isolated singularity in an 8-dimensional toric variety with the following $\mathbb{Q}$-stratification
	\begin{align*}
	\mathcal{X}=\mathcal{X}_{7} \supset \mathcal{X}_{1}.
	\end{align*}
	Let $\mathcal{L}$ be the link of a connected component of $\mathcal{X}_{1}$ and $f_{1}^{\mathcal{L}}$ be the number of 1-dimensional faces of the underlying polytope of $\mathcal{L}$, $\mathcal{P}$. Then, the Betti numbers of $\mathcal{L}$ are
	\begin{align*}
	b_{5}^{\mathcal{L}}=1,\;		b_{4}^{\mathcal{L}}=0, \;
	b_{3}^{\mathcal{L}}=f_{1}^{\mathcal{L}}-3, \;
	b_{2}^{\mathcal{L}}=f_{1}^{\mathcal{L}}-3, \;
	b_{1}^{\mathcal{L}}=0, \;	
	b_{0}^{\mathcal{L}}&=1.
	\end{align*}
\end{lemma}
\begin{proof}
Let $\pi: \mathcal{X} \longrightarrow \mathcal{M}$ be the natural projection. Note that $\mathcal{X}_{1} = \bigsqcup_{i} S^{1}_{i}$. Let $\mathcal{L}$ be the link of $S^{1} \subset \mathcal{X}_{1}$, and let $\mathcal{P} \subset \mathcal{M}$ be an embedding, as in the construction of links in \cite[Section 3]{banagl2024link}. In other words, we choose $\mathcal{P}$ to be the intersection of a 2-dimensional plane with $\mathcal{M}$, ensuring that it intersects all the 1-dimensional faces of $\mathcal{M}$ for which $\pi(S^{1})$ is a proper face. Then, we have $\mathcal{L} \cong \sfrac{\pi^{-1}(\mathcal{P})}{S^{1}}$. Hence, we obtain $\mathcal{L}$ by attaching $T^3$, $T^2$, and $S^1$ to the 2-dimensional and 1-dimensional faces and vertices of $\mathcal{P}$, respectively. The attaching maps are the corresponding collapses derived from the cone associated with the isolated singularity in the 8-dimensional toric variety. Following \cite[Proposition 33]{banagl2024link}, we can form the boundary operators and calculate the Betti numbers.

\end{proof}

\begin{lemma}\label{MHOM}
	Let $M$ be the $\mathbb{Q}$-manifold obtained by cutting out all singularities of $\mathcal{X}$, where $\mathcal{X}$ is a link of an isolated singularity in an 8-dimensional toric variety. Then, the lower Betti numbers of $M$ are
	\begin{align*}
	b_{4}^{M}&=3f_{1}-f_{2}-6-(b_{1}+b_{2}), \,
	b_{3}^{M}=3f_{1}-f_{2}-6-b_{1}, \,
	b_{2}^{M}=f_{1}-4, \,
	b_{1}^{M}=0, \,
	b_{0}^{M}=1,
	\end{align*}
	where $f_{1}$and $f_{2}$ are the number of 2-dimensional and 1-dimensional faces, respectively, of $\mathcal{M}$, the underlying 3-dimensional polytope of $\mathcal{X}$. The parameters $b_{2}$ and $b_{1}$ are defined in Equations (\ref{bp1}), and (\ref{bp2}).
\end{lemma}
\begin{proof}
	The space $M$ is a $\mathbb{Q}$-manifold satisfying Lefschetz duality rationally. Thus, we have
	\begin{align*}
	\text{Hom}(H^{k}(M;\mathbb{Q});\mathbb{Q}) \cong H_{7-k}(M, \partial M ; \mathbb{Q})
	\end{align*}
	From the CW structure of $M$ it is clear that
	\begin{align}\label{MXhomology}
	H_{i}(X) \cong H_{i}( \sfrac{M}{ \partial M } )\; \; \text{for} \;\; i \geq 3
	\end{align}
    because coning off the boundary, $\partial M$, only modifies the lower dimensional cells. Finally, we use $\widetilde{H}_{j}(\sfrac{M}{ \partial M }) \cong \widetilde{H}_{j}(M, \partial M)$.  
\end{proof}
At last, we possess all the essential tools to establish the first main theorem of this section. In the proof of the following theorem, we will show that $b_{1} = 0$. With $b_{1} = 0$ established, we can provide a relatively simple algebraic description of the parameter $b_{2}$ using the long exact sequence of the relative homology groups.
\begin{theorem}\label{IntHom}
	Let $\mathcal{X}$ be the link of an isolated singularity in an 8-dimensional toric variety. Let $I^{\bar{n}}\mathcal{X}$ be the generalized intersection space associated with $\mathcal{X}$ with respect to middle perversity $\bar{n}$, as defined in Definition \ref{IXgen}. Then, the Betti numbers of $I^{\bar{n}}\mathcal{X}$ are given by
	\begin{align*}
	b_{7}^{I^{\bar{n}}\mathcal{X}}&=0,\; 
	b_{6}^{I^{\bar{n}}\mathcal{X}}=m-1, \;
	b_{5}^{I^{\bar{n}}\mathcal{X}}=f_{2}-4-b_{2}, \;
	b_{4}^{I^{\bar{n}}\mathcal{X}}=3f_{1}-f_{2}-6-b_{2}, \\
	b_{3}^{I^{\bar{n}}\mathcal{X}}&=3f_{1}-f_{2}-6-b_{2},  \;
	b_{2}^{I^{\bar{n}}\mathcal{X}}=f_{2}-4-b_{2}, \;
	b_{1}^{I^{\bar{n}}\mathcal{X}}=m-1, \;
	b_{0}^{I^{\bar{n}}\mathcal{X}}=1, 
	\end{align*}
	where $f_{1}$ and $f_{2}$ are the numbers of 2-dimensional and 1-dimensional faces of the underlying 3-dimensional polytope of $\mathcal{X}$, $\mathcal{P}$. The parameter $m$ denotes the number of singular components of $\mathcal{X}_{1}$, and we defined $b_{2}$ in Equation (\ref{bp2}).
\end{theorem}
\begin{proof}
	For the proof, we use the Mayer-Vietoris sequence. Let $M$ be the $\mathbb{Q}$-manifold obtained by cutting out the singularities of $\mathcal{X}$. We set $\mathcal{U} \simeq M$. Let $\mathcal{V} \cong \mathcal{C}(\sqcup_{i} S^{1} \times (\mathcal{L}_{i})_{<k} )$, where $i$ goes over the singular connected components of $\mathcal{X}_{1}$. Hence, we get $\cap := \mathcal{U} \cap \mathcal{V} \simeq \sqcup_{i} S^{1} \times (\mathcal{L}_{i})_{<k}$, and $I^{\bar{n}}\mathcal{X} = \mathcal{U} \cup \mathcal{V}$. Consequently, we have
	\begin{align*}
		&0 \longrightarrow H_{6}(M)  \xrightarrow{\phantom{-} \cong \phantom{-} } H_{6}(I^{\bar{n}}\mathcal{X})  \longrightarrow 0 \\
		&0 \longrightarrow H_{5}(M)  \xrightarrow{\phantom{-} \cong \phantom{-}} H_{5}(I^{ \bar{n} }\mathcal{X})  \longrightarrow 0 \\
		&0 \longrightarrow  H_{1}(I^{ \bar{n} }\mathcal{X}) \xrightarrow{ \phantom{-} \cong \phantom{-}} \overbrace{ \widetilde{H_{0}}(\cap) }^{\cong \mathbb{Q}^{m-1 } } \xrightarrow{\phantom{--}}   0.
	\end{align*}
	Thus, we arrive at $b_{1}^{I \mathcal{X}}=m-1$. Since the singularities of $\mathcal{X}$ are even-co-dimensional, we have $I^{ \bar{n} }\mathcal{X}=I^{ \bar{m} }\mathcal{X}$, and duality yields 
	$b_{6}^{I \mathcal{X}}= m-1$. 
	The more interesting part of the long exact sequence is the following:
	\begin{align*}
		0 &\longrightarrow  \overbrace{H_{4}(M)}^{\cong \mathbb{Q}^{3f_{1}-f_{2}-6-(b_{1}+b_{2})} }  \longrightarrow  H_{4}(I^{ \bar{n} }\mathcal{X})
		\longrightarrow \overbrace{ H_{3}(\cap) }^{\cong \mathbb{Q}^{3f_{1}-f_{2}-6 } } \longrightarrow  \overbrace{H_{3}(M)}^{\cong \mathbb{Q}^{3f_{1}-f_{2}-6-b_{1}} } \\&  \longrightarrow H_{3}(I^{ \bar{n} }\mathcal{X}) \longrightarrow \overbrace{ H_{2}(\cap) }^{\cong \mathbb{Q}^{3f_{1}-f_{2}-6 } } \longrightarrow   \overbrace{H_{2}(M)}^{\cong \mathbb{Q}^{f_{1}-4} }  \longrightarrow H_{2}(I^{ \bar{n} }\mathcal{X}). 
	\end{align*}
	
	Note that for a non-singular connected component of $S^{1}_{i} \subset \mathcal{X}_{1}$, we have $f_{1}^{i}=3$. Using the well-known Euler formula for 3-dimensional polytopes implies $\sum_{i=1}^{m}(f_{1}^{i}-3)=	\sum_{i=1}^{f_{2}-f_{1}+2}(f_{1}^{i}-3)$. Counting the 1-dimensional faces of the underlying polytope of the link of $S^{1}_{i}$, $\mathcal{P}_{i}$, for each connected component $S^{1}_{i}$ equals counting the neighboring 1-dimensional faces of each vertex in the underlying polytope of $\mathcal{X}$, $\mathcal{M}$. Since each 1-dimensional face of the polytope $\mathcal{M}$ has precisely two vertices in its topological boundary, we have $\sum_{i=1}^{f_{2}-f_{1}+2}f_{1}^{i}=2 f_{2}$. Thus, we get $\sum_{i=1}^{f_{2}-f_{1}+2}(f_{1}^{i}-3)=2f_{2}-3(2-f_{1}+f_{2})=3 f_{1}-f_{2} -6$. Since $\mathcal{L}$ is rationally 3-segmented, we have
	\begin{align*}
		H_{i}(\partial M) \cong H_{i}(S^{1} \times \mathcal{L}_{<3} ) \quad \text{for} \quad i < 3.
	\end{align*}
	The definition yields $H_{i}(M)\cong H_{i}(\mathcal{U})$. Consider the above Mayer-Vietoris exact sequence and the exact sequence of the relative homology of the pair $( M ,\partial M)$ in the following diagram:
	\begin{align}\label{5lemma}
		\xymatrix{
			H_{3}( \partial M)  \ar[r] & H_{3}(M)  \ar[r] & H_{3}(M, \partial M) \ar[r] & H_{2}(M) \ar[r]  & H_{2}(\partial M)\\
			H_{3}(\cap) \ar[r] \ar[u] & H_{3}(M) \ar[r] \ar[u] & H_{3}(I^{ \bar{n} }\mathcal{X}) \ar[u] \ar[r] & H_{2}(M) \ar[u] \ar[r]& H_{2}( \cap). \ar[u] 
		}
	\end{align}
	The commutativity of the diagram
	\begin{align*}
		\xymatrix{ \partial M  \ar@{^{(}->}[r] & M \\
			\cap  \ar@{^{(}->}[u]  \ar@{^{(}->}[ru]   & \\}
	\end{align*}
	ensures that the outer left and right squares in Diagram \ref{5lemma} commute. From \cite[Lemma 2.46]{banagl2010intersection}, it follows that there is a homomorphism $H_{3}(I^{ \bar{n} }\mathcal{X}) \longrightarrow H_{3}(M, \partial M)$ that makes Diagram \ref{5lemma} commutative. Using the 5-lemma yields $b_{3}^{I\mathcal{X}}=b_{3}^{( M ,\partial M)}=b_{4}^{M}$, whereas, in the last equality, we used Lefschetz duality. By using the duality of Betti numbers of intersection spaces and Lemma \ref{MHOM}, we get $b_{4}^{M}=b_{3}^{I\mathcal{X}}=b_{4}^{I\mathcal{X}}=3f_{1}-f_{2}-6-(b_{1}+b_{2})$. Employing the exactness at $H_{3}(\cap) \longrightarrow H_{3}(M)$ and considering that $0 \leq b_{1}$, it immediately follows that $b_{1}=0$, where $b_{1}$ is defined in Equation (\ref{bp1}). Finally, we get $b_{2}^{I\mathcal{X}}=f_{2}-2-b_{2}$.
	
\end{proof}

We initially defined $b_{2}$ using boundary operators. Primarily, forming boundary operators requires a CW structure on the link $\mathcal{X}$. The following lemma aims to give an algebraic description of $b_{2}$. With this at hand, we can present the second main theorem of this section.

\begin{lemma}\label{b2alg}
	Let $\mathcal{X}$ be the link of an isolated singularity in an 8-dimensional toric variety, and let $(M, \partial M)$ be the $\mathbb{Q}$-manifold obtained by cutting out the singularities of $\mathcal{X}$. Let $b_2$ be the parameter defined in Equation (\ref{bp2}), and let $i: \partial M \xhookrightarrow[]{ \;\;\;\; } M$ be the inclusion. Then, we have
	\begin{align*}
		b_{2} = \operatorname{rk}(\operatorname{ker}(H_{3}(\partial M) \xrightarrow{ \; \;  i_{\ast}  \; \; } H_{3} (M)  )).
	\end{align*}
\end{lemma}
\begin{proof}
 We consider the long exact sequence of the relative homology of $(M, \partial M)$. Using rational Lefschetz duality and Equation (\ref{MXhomology}), we obtain
	\begin{align*}
		0 \longrightarrow \overbrace{H_{4}(M)}^{3 f_{1} - f_{2}-6-b_{2} } \longrightarrow \overbrace{ H_{4}(M, \partial M)  }^{3 f_{1} -f_{2} - 6  } \longrightarrow \overbrace{H_{3} (\partial M) }^{3 f_{1} -f_{2} -6 } \longrightarrow \overbrace{H_{3}(M) }^{3f_{1} -f_{2}-6 }  \longrightarrow \overbrace{H_{3}(M, \partial M)}^{3f_{1} -f_{2}-6- b_{2} }.
	\end{align*}
	By employing the exactness of the above sequence, we arrive at
	$b_{2} = \operatorname{rk}(\operatorname{ker}(H_{3}(\partial M) \longrightarrow H_{3} (M)))$.
\end{proof}

\begin{theorem}\label{MainTheo2}
	Let $\mathcal{X}$ be the link of an isolated singularity in an 8-dimensional toric variety. Let $(M , \partial M)$ be the $\mathbb{Q}$-manifold obtained by cutting out all 1-dimensional singularities of $\mathcal{X}$, and let $i: \partial M \xhookrightarrow[]{ \;\;\;\; } M$ be the inclusion map. Then, the Betti numbers of $\mathcal{X}$ are
	\begin{align*}
		b^{\mathcal{X}}_{7}=1, \;	b^{\mathcal{X}}_{6}=0, \;	&b^{\mathcal{X}}_{5}=f_{1}-4,\;	b^{\mathcal{X}}_{4}=3f_{1}-f_{2}-6, \;
		b^{\mathcal{X}}_{3}=3f_{1}-f_{2}-6-b_{2}, \\	&b^{\mathcal{X}}_{2}=f_{1}-4-b_{2}, \;	b^{\mathcal{X}}_{1}=0, \; b^{\mathcal{X}}_{0}=1,
	\end{align*}
	where 	$b_{2} = \operatorname{rk}(\operatorname{ker}(H_{3}(\partial M) \xrightarrow{ \; \;  i_{\ast}  \; \; } H_{3} (M)  ))$.
\end{theorem}
\begin{proof}
	In the proof of Theorem \ref{IntHom}, we showed that $b_{1}=0$. Consequently, the claim follows from Proposition \ref{Hom1Ver} and Lemma \ref{b2alg}. 
\end{proof}

\section{$S^{1}$-bundle and homotopy groups of the links}

Consider the stratification on a toric variety $ X_{\Sigma} $, associated with a complete fan $ \Sigma $, induced by a moment map, as introduced in \cite{banagl2024link}. The links of the strata are odd-dimensional topological pseudomanifolds and, therefore, cannot be studied through the lens of algebraic geometry. However, we use the methods introduced in \cite[Chapter 3.3]{cox2024toric} to show that, for a given link, there is a fiber bundle with fiber $\mathbb{C}^{\ast}$, where the total space and the base space are toric varieties. A restriction of this fiber bundle induces a fiber bundle over the same base space, with the given link as the total space and $S^{1}$ as the fiber. For a detailed introduction to fiber bundles, the reader may consult \cite{husemoller1966fibre}.

To be able to study morphisms between toric varieties, we need the following definition.
\begin{definition}\label{DefComp}
	Let $ N_1 $, $ N_2 $ be two lattices with $ \Sigma_1 $ a fan in $ (N_1)_\mathbb{R} $ and $ \Sigma_2 $ a fan in $ (N_2)_\mathbb{R} $. A $ \mathbb{Z} $-linear mapping $ \overline{\phi} : N_1 \longrightarrow N_2 $ is \textbf{compatible} with the fans $ \Sigma_1 $ and $ \Sigma_2 $ if for every cone $ \sigma_1 \in \Sigma_1 $, there exists a cone $ \sigma_2 \in \Sigma_2 $ such that $ \overline{\phi}_\mathbb{R}(\sigma_1) \subseteq \sigma_2 $.
\end{definition}
A compatible $\mathbb{Z}$-linear map $ \overline{\phi}: N_1 \to N_2 $, in the sense described above, induces a morphism between toric varieties $ \phi: X_{\Sigma_1} \longrightarrow X_{\Sigma_2} $. Let $ \nu \in X_{\Sigma} $ be an isolated or singular point in the toric variety $ X_{\Sigma} $, associated with the complete fan $ \Sigma \subseteq N $. Let $ \Sigma_{\nu} $ be the dual cone to $ \nu $ in $ N $. One can extend the topological construction of compact toric varieties to the case where the underlying fan is not complete. We will not carry out this construction in full detail here but will instead provide an overview of how it works, focusing at least on the case of a fan consisting of a single proper cone. Reversing the inclusions and the dimensional grading of the faces of $\Sigma_{\nu}$ results in an abstract object $\mathcal{P}_{\nu}$ that can be geometrically realized as an open cone over a $(\dim(\Sigma_{\nu}) - 1)$-dimensional polytope $\mathcal{M}_{\mathcal{L}_{\nu}}$. The polytope $\mathcal{M}_{\mathcal{L}_{\nu}}$ is the underlying polytope of the link of the point $\nu $, denoted by $\mathcal{L}_{\nu}$. The point $\nu \in X_{\Sigma_{\nu}} \subset X_{\Sigma}$, considered as an element of $\mathcal{P}_{\nu}$, is dual to the interior of $\Sigma_{\nu}$. One can show that $ X_{\Sigma_{\nu}} \cong \mathring{\mathcal{C}}(\mathcal{L}_{\nu}) $, where $ \mathring{\mathcal{C}} $ denotes the open cone over the link $ \mathcal{L}_{\nu} $. Let $ \mathring{\Sigma}_{\nu} $ be the fan obtained from $ \Sigma_{\nu} $ by removing its interior. Hence, we have $ X_{\mathring{\Sigma}_{\nu}} \cong \mathring{\mathcal{C}}(\mathcal{L}_{\nu}) \setminus \nu \cong \mathcal{L}_{\nu} \times \mathbb{R} $. Let $ N_0 \subset N $ be the sublattice spanned by a 1-dimensional cone that lies in the interior of $ \Sigma_{\nu} $. Let also $ (N')_{\mathbb{R}} = (N_0)_{\mathbb{R}}^{\perp} $. Consider the natural projection $ \pi^{\prime}: N \longrightarrow N' $, and define $ \Sigma'_{\nu} = \pi(\mathring{\Sigma}_{\nu}) $. The toric variety $ X_{\Sigma'_{\nu}} $ is compact since $ \Sigma'_{\nu} \subseteq N' $ is a complete fan. Note that $ \Sigma_0 = \{ \sigma \in \mathring{\Sigma}_{\nu} \mid \sigma \subseteq (N_0)_{\mathbb{R}} \} = \{ 0 \} $. We have the following short exact sequence:
\begin{align*}\label{LongEx}
0 \rightarrow N_0 \hookrightarrow N \xrightarrow{\pi} N' \rightarrow 0.
\end{align*}

\begin{definition}
	 In the situation of Definition \ref{DefComp}, we say $\Sigma$ is split by $\Sigma^{\prime}$ and $\Sigma_0$ if there exists a subfan $\hat{\Sigma} \subseteq \Sigma$ such that:
	\begin{itemize}
\item[(a)] $\overline{\phi}_\mathbb{R}$ maps each cone $\hat{\sigma} \in \hat{\Sigma}$ bijectively to a cone $\sigma^{\prime} \in \Sigma^{\prime}$ such that $\overline{\phi}( \hat{ \sigma } \cap N) = \sigma^{\prime} \cap N^{\prime}$. Furthermore, the map $\hat{\sigma} \mapsto \sigma$ defines a bijection $\Sigma^{\prime} \to \hat{\Sigma}$. 

\item[(b)] Given cones $\hat{\sigma} \in \hat{\Sigma}$ and $\sigma_0 \in \Sigma_0$, the sum $\hat{\sigma} + \sigma_0$ lies in $\Sigma$, and every cone of $\Sigma$ arises this way.
	\end{itemize}
\end{definition}
With the above notation, we immediately arrive at the following lemma: 
\begin{lemma}
	The fan $\mathring{\Sigma}_{\nu}$ is split by $\Sigma_{\nu}^{\prime}$ and $\Sigma_0= \{ 0\}$.
\end{lemma}
\begin{theorem}\label{HomtGr}
Let $ X_{\Sigma} $ be a toric variety associated with the complete fan $ \Sigma $. Consider the stratification induced by a moment map, and let $ \mathcal{L}_{x} $ be the link of the point $ x \in X_{\Sigma} $. Then, there exists a fiber bundle $ \pi: \mathcal{L}_{x} \longrightarrow X_{\Sigma^{\prime}} $ with fiber $ S^{1} $, where $ X_{\Sigma^{\prime}} $ is a compact toric variety. Furthermore, we obtain the fan $ \Sigma^{\prime} $ as follows:

Let $ \Sigma_{x} \in \Sigma $ be the corresponding cone to the orbit in which $ x $ lies. The complete fan $ \Sigma^{\prime} $ is obtained by projecting the topological boundary of $ \Sigma_{x} \subset N := \operatorname{span}(\Sigma_{x}) \cap \mathbb{Z}^{\dim(\Sigma) } \cong \mathbb{Z}^{\dim ( \Sigma_{x} )  }  $ along a 1-dimensional cone $ \tau $ lying in the interior of $ \Sigma_{x} $ into the lattice $N_{0}^{\perp} := \operatorname{span} (\tau^{\perp}) \cap \mathbb{Z}^{\dim ( \Sigma )} \cong \mathbb{Z}^{\dim (\Sigma_{x})-1} $.
\end{theorem}
\begin{proof}
   	First, we prove the statement for a point in the bottom stratum of $X_{\Sigma}$. We use the same notation as in the above discussion. Using the short exact sequence \ref{LongEx} and \cite[Theorem 3.3.19]{cox2024toric}, we conclude that $X_{\mathring{\Sigma}_{\nu}}$ is a fiber bundle over $X_{\Sigma^{\prime}_{\nu} }$ with fiber $X_{\Sigma_{0}}$, where $\mathring{\Sigma}_{\nu} \subseteq N$, $\Sigma^{\prime}_{\nu} \subset N^{\prime}$, and $\Sigma_{0}\subseteq N_{0}$.
Thus, we arrive at the following fiber bundle:
\begin{align}
	\mathbb{R} \times S^{1} \xrightarrow{g=(\operatorname{id}_{\mathbb{R}},\overline{g})} \mathbb{R} \times \mathcal{L}_{\nu} \xrightarrow{f=(\operatorname{pr},\overline{f})} \ast \times X_{\Sigma^{\prime}_{\nu} },
\end{align}
where the maps $\operatorname{id}_{\mathbb{R}}$ and $\operatorname{pr}$ denote the identity map and the projection to a point, respectively. It follows that 
\begin{align}
 S^{1} \xrightarrow{\overline{g}} \mathcal{L}_{\nu} \xrightarrow{\overline{f}}   X_{\Sigma^{\prime}_{\nu} }
\end{align}
is also a fiber bundle. This completes the proof for the case where $ x $ lies in the bottom stratum of $ X_{\Sigma} $. For an arbitrary point $ x $, let $ \Sigma_{x} \in \Sigma $ be the corresponding cone to the orbit in which $ x $ lies. Consider $ \Sigma_{x} $ as a cone in the sublattice $ N := \operatorname{span}(\Sigma_{x}) \cap \mathbb{Z}^{\dim(\Sigma) } \cong \mathbb{Z}^{\dim ( \Sigma_{x} )  } $. Hence, we arrive at $X_{\Sigma_{x} } \cong \mathring{\mathcal{C}} (\mathcal{L}_{x})$, where $\Sigma_{x} \subset N$. Let $\mathring{\Sigma}_{x} \subset N$ be the fan obtained from $\Sigma_{x}$ by removing its interior. It follows that $X_{\mathring{\Sigma}_{x}} \cong \mathcal{L}_{x} \times \mathbb{R}$. The fan $\mathring{\Sigma}_{x}$ splits by $\Sigma_{0} = \{ 0\}$ and $\Sigma^{\prime}_{x} := \pi^{\prime} (\mathring{\Sigma}_{x})$. The map $\pi^{\prime} : N \longrightarrow  ( \sfrac{N_{\mathbb{R}}}{ (N_{0})_{\mathbb{R}} } ) \cap N $ is defined where $N_0 \subset N$ is the lattice spanned by a 1-dimensional cone lying in the interior of $\Sigma_{x}$. The rest of the proof follows the same lines as in the case of the points in the bottom stratum.
\end{proof}

\begin{corollary}
	In the situation of Theorem \ref{HomtGr}, we have the exact sequence
	$
		0 \longrightarrow \pi_{2}(\mathcal{L}_{x}) \longrightarrow \pi_{2}(X_{\Sigma^{\prime} }) \longrightarrow \mathbb{Z} \longrightarrow \pi_{1}( \mathcal{L}_{x} ) \longrightarrow \pi_{1}(X_{\Sigma^{\prime}}) \longrightarrow 0,
	$
and $ \pi_ {i}(\mathcal{L}_{x} ) \cong \pi_{i}(X_{\Sigma^{\prime} })$ for $ i \geq 3$.
\end{corollary}

\begin{proof}
We use the long exact sequence of homotopy groups induced by $S^{1} \hookrightarrow \mathcal{L}_{x} \rightarrow X_{\Sigma^{\prime}}$.
\end{proof}
The cohomology groups of the base space and the total space of a sphere bundle are closely related. Chern and Spanier in \cite{chern1950homology} introduced an exact cohomology sequence for sphere bundles. Naturally, there also exists an exact sequence of homology groups for a given sphere bundle. For a given $S^d$
-bundle, $f:X \longrightarrow B$, the long exact sequence of homology groups is given by
\begin{align}
\dots \xrightarrow{f_{\ast}} H_{p+1}(B) \xrightarrow{\Psi^{\prime}}H_{p-d}(B)
 \xrightarrow{\Phi^{\prime}} H_{p}(X) \xrightarrow{f_{\ast}} H_{p}(B) \xrightarrow{\Psi^{\prime}} \dots .
\end{align}
The homomorphisms $f_{\ast}$ and $\Phi^{\prime}$ are explicitly defined in \cite{chern1950homology}. Since their exact form is not needed for our purposes, we omit their definitions. The definition of $\Psi^{\prime}:H_{p}(B) \longrightarrow H_{p-d-1}(B)$ will be provided later. Our goal is to prove the following proposition.
\begin{proposition}\label{BettiComp}
	Let $\mathcal{L}_{x}$ be the link of an isolated point in a toric variety of real dimension 8. Let $\pi: \mathcal{L}_{x} \longrightarrow X_{\Sigma}$ be an $S^{1}$-bundle as in Theorem \ref{HomtGr}. Then the non-combinatorial invariant parameter $b$ in the Betti numbers $b_{2}^{X_{\Sigma}}$ and $b_{3}^{X_{\Sigma}}$ introduced by McConnell in \cite{mcconnell1989rational} or \cite{banagl2024link} equals the parameter $b_{2}$ in the Betti numbers of $\mathcal{L}_{x}$ as defined in Theorem \ref{MainTheo2}.
\end{proposition}

\begin{proof}
    We begin by studying the aforementioned long exact sequence of homology groups. Since $\dim(\mathcal{L}_{x})=7$, we obtain
    \begin{align*}
    	0 &\rightarrow H_{6}(X_{\Sigma}) \rightarrow H_{7}(\mathcal{L}_{x}) \rightarrow 0 \rightarrow \overbrace{H_{5}(X_{\Sigma})}^{=0} \rightarrow H_{6}(\mathcal{L}_{x}) \rightarrow H_{6}(X_{\Sigma}) \\
    	&\rightarrow H_{4}(X_{\Sigma}) \rightarrow H_{5}(\mathcal{L}_{x}) \rightarrow \overbrace{H_{5}(X_{\Sigma})}^{=0} \rightarrow H_{3}(X_{\Sigma}) \rightarrow H_{4}(\mathcal{L}_{x}) \rightarrow H_{4}(X_{\Sigma}) \\
    	&\rightarrow H_{2}(X_{\Sigma}) \rightarrow H_{3}(\mathcal{L}_{x}) \rightarrow H_{3}(X_{\Sigma}) \rightarrow \overbrace{H_{1}(X_{\Sigma})}^{0} \rightarrow H_{2}(\mathcal{L}_{x}) \rightarrow H_{2}(X_{\Sigma}) \\
    	&\rightarrow H_{0}(X_{\Sigma}) \rightarrow H_{1}(\mathcal{L}_{x}) \rightarrow \overbrace{H_{1}(X_{\Sigma})}^{=0} \rightarrow 0 \rightarrow H_{0}(\mathcal{L}_{x}) \rightarrow H_{0}(B) \rightarrow 0.
    \end{align*}
    Hence, we arrive at the following exact sequences.
    \begin{align*}
    &0 \rightarrow H_{6}(X_{\Sigma}) \rightarrow H_{7}(\mathcal{L}_{x}) \rightarrow 0	\\
    	&0 \rightarrow H_{6}(\mathcal{L}_{x}) \rightarrow H_{6}(X_{\Sigma}) \rightarrow H_{4}(X_{\Sigma}) \rightarrow H_{5}(\mathcal{L}_{x}) \rightarrow 0 \\
    	&0 \rightarrow H_{3}(X_{\Sigma}) \rightarrow H_{4}(\mathcal{L}_{x}) \rightarrow H_{4}(X_{\Sigma}) 
    	\rightarrow H_{2}(X_{\Sigma}) \rightarrow H_{3}(\mathcal{L}_{x}) \rightarrow H_{3}(X_{\Sigma}) \rightarrow 0 \\
    	&0 \rightarrow H_{2}(\mathcal{L}_{x}) \rightarrow H_{2}(X_{\Sigma}) 
    	\rightarrow H_{0}(X_{\Sigma}) \rightarrow H_{1}(\mathcal{L}_{x}) \rightarrow 0\\
        &0 \rightarrow H_{0}(\mathcal{L}_{x}) \rightarrow H_{0}(B) \rightarrow 0.
    \end{align*}
    The exact sequences above yield the following equalities between Betti numbers.
    \begin{align*}
     b_{7}^{\mathcal{L}_{x}}=1, \; b_{5}^{\mathcal{L}_{x}} = b_{4}^{X_{\Sigma}} -1, \; b_{4}^{\mathcal{L}_{x}} -b_{3}^{\mathcal{L}_{x}} = b_{4}^{X_{\Sigma}} -b_{2}^{X_{\Sigma}}, \; b_{2}^{\mathcal{L}_{x}} =b_{2}^{X_{\Sigma}}-1,\; b_{0}^{\mathcal{L}_{x}} = b_{0}^{X_{\Sigma}},
    \end{align*}
    where we used $b_{6}^{\mathcal{L}_{x}}=b_{1}^{\mathcal{L}_{x}}=0$ and $b_{6}^{X_{\Sigma}}=b_{0}^{X_{\Sigma}}=1$. Note that from the construction of the projection map $\pi^{\prime}$, defined in the proof of Theorem \ref{HomtGr}, the numbers of 1-dimensional and 2-dimensional faces of the underlying polytopes of $\mathcal{L}_{x}$ and $X_{\Sigma}$ are equal. Hence, the numbers $f_{1}^{\mathcal{L}_{x}}$ and $f_{2}^{\mathcal{L}_{x}}$, defined in Theorem \ref{MainTheo2}, equal the numbers $f_{1}^{X_{\Sigma}}$ and $f_{2}^{X_{\Sigma}}$, which are, respectively, the numbers of 1-dimensional and 2-dimensional cones in the underlying fan of $X_{\Sigma}$. Therefore, by using the explicit form of the Betti numbers of $X_{\Sigma}$ provided in \cite{mcconnell1989rational} or \cite{banagl2024link}, the claim follows. 
\end{proof}
By reviewing the proof of the above theorem we arrive at the following result.
\begin{theorem}
		Let $\mathcal{L}_{x}$ be the link of an isolated point in a toric variety of real dimension 8. Let $\pi: \mathcal{L}_{x} \longrightarrow X_{\Sigma}$ be an $S^{1}$-bundle, as in Theorem \ref{HomtGr}. Let $\Omega \in H^{2}(X_{\Sigma})$ denote the characteristic cohomology class of this $S^{1}$-bundle, as defined in \cite[Theorem 3.3]{chern1950homology}. Then, the non-combinatorial invariant parameter $b$ in the Betti numbers of $X_{\Sigma}$, defined in \cite{mcconnell1989rational} or \cite{banagl2024link}, is given by
		\begin{align*}
			b=\operatorname{rk}\big(\operatorname{ker}(\Psi^{\prime}:H_{4}(X_{\Sigma}) \longrightarrow H_{2}(X_{\Sigma}))\big),
		\end{align*}
		where $\Psi^{\prime}(z)= \Omega \cap z$ for all $z \in H_{4}(X_{\Sigma})$.
\end{theorem}
\begin{proof}
	The homomorphisms $\Psi^{\prime}:H_{p}(B) \longrightarrow H_{p-d-1}(B)$ associated to an $S^{d}$-bundle have been studied in \cite{chern1950homology}. For the $S^{1}$-bundle $\pi:\mathcal{L}_{x} \longrightarrow X_{\Sigma}$, the claim follows from the exactness of sequence
	\begin{align*}
		0 \rightarrow H_{3}(X_{\Sigma}) \rightarrow H_{4}(\mathcal{L}_{x}) \rightarrow H_{4}(X_{\Sigma}) 
		\rightarrow H_{2}(X_{\Sigma}) \rightarrow H_{3}(\mathcal{L}_{x}) \rightarrow H_{3}(X_{\Sigma}) \rightarrow 0.
	\end{align*} 
	More precisely, one obtains $\operatorname{rk}\big(\operatorname{ker}(H_{4}(X_{\Sigma}) \longrightarrow H_{2}(X_{\Sigma}))\big)=b_{4}^{\mathcal{L}_{x}}-b_{3}^{X_{\Sigma}}$. On the other hand, $b_{4}^{\mathcal{L}_{x} }=3 f_{1}^{\mathcal{L}_{x} }-f_{2}^{\mathcal{L}_{x}}-6$ and $b_{3}^{X_{\Sigma} }=3 f_{1}^{X_{\Sigma} }-f_{2}^{X_{\Sigma}}-6-b$. The numbers $f_{1}$ and $f_{2}$ denote the number of 2-dimensional and 1-dimensional faces of the underlying polytopes of the corresponding space. As noted in the proof of Theorem \ref{BettiComp}, $f_{1}^{\mathcal{L}_{x} }=f_{1}^{X_{\Sigma} }$ and $f_{2}^{\mathcal{L}_{x} }=f_{2}^{X_{\Sigma} }$. The claim then follows.
\end{proof}


\begin{thebibliography}{99}
	
	\bibitem{banagl2024link}Banagl, Markus and Ghaed Sharaf, Shahryar {\em Link Bundles and Intersection Spaces of Complex Toric Varieties}, Journal of Singularities \textbf{28} (2025), 55 - 103
	
	
	
	
	\bibitem{banagl2010intersection} Banagl, M. 
	{\em Intersection spaces, spatial homology truncation, and string theory}, 
	Lecture Notes in Mathematics \textbf{1997}. (Springer-Verlag
	Berlin - Heidelberg, 2010)
	
		\bibitem{cheeger1980hodge} Cheeger, J. 
	{\em On the Hodge theory of Riemannian pseudomanifolds}, 
	in Proc. of Symposia in Pure Math, \textbf{36} (1980), pp. 91--146,
	American Math. Soc. 
	
	
	\bibitem{cheeger1979spectral}Cheeger, J. {\em On the spectral geometry of spaces with cone-like singularities},  Proceedings Of The National Academy Of Sciences, \textbf{76} (1979), pp. 2103--2106 
	
	\bibitem{cheeger1983spectral} Cheeger, J. {\em Spectral geometry of singular Riemannian spaces}, Journal Of Differential Geometry, \textbf{18} (1983), pp. 575--657 
	
	\bibitem{chern1950homology} Chern, Shiing-Shen and Spanier, E {\em The homology structure of sphere bundles},
	Proceedings of the National Academy of Sciences \textbf{18} (1950), pp. 248--255


	\bibitem{cox2024toric} Cox, David A and Little, John B and Schenck, Henry K {\em Toric varieties}, American Mathematical Society \textbf{124} (2024)
	
	
	\bibitem{goresky1980intersection}Goresky, M. \& MacPherson, R. {\em Intersection homology theory},  Topology, \textbf{19}  (1980), pp. 135--162
	
	\bibitem{husemoller1966fibre} Husem{\"o}ller, D  {\em Fibre bundles}, Springer, Vol. \textbf{5} (1966)
	
	
	
	
	\bibitem{mcconnell1989rational}McConnell, M. {\em The rational homology of toric varieties is not a combinatorial invariant},  Proceedings of the American Mathematical Society, \textbf{105} (1989), pp. 986--991 
		
	\bibitem{rourkesanderson} Rourke, C., Sanderson, B.,
	{\em Homology stratifications and intersection homology},
	Geometry \& Topology Monographs, \textbf{2} (1999), pp. 455--472	
	
	
	\bibitem{wrazidlo2013induced} Wrazidlo, D. {\em Induced maps between intersection spaces}, Master thesis, (Universität Heidelberg, 2013)
	

\end{thebibliography}
\end{document}